\documentclass[11pt]{article}
\usepackage{amsfonts,amscd,amssymb, amsmath}
\usepackage[usenames]{color}

\usepackage{graphicx}

\usepackage{euscript}
\usepackage[all,tips]{xy}
\usepackage{epic,eepic}

\usepackage{color}
\usepackage{tikz}

\newtheorem{definition}{Definition}[section]

\newtheorem{proposition}[definition]{Proposition}
\newtheorem{corollary}[definition]{Corollary}
\newtheorem{remark}[definition]{Remark}

\newtheorem{theorem}[definition]{Theorem}
\newtheorem{example}[definition]{Example}

\def\rawo\lonra{\longrightarrow}

\def\ot{\otimes}

\allowdisplaybreaks[4]
\newenvironment{proof}{{\it Proof.}}{\hfill $ \square $ \vskip 4mm}

\newcommand{\BiHomMod}{\mathbf{BiHomMod}}
\newcommand{\BiHomNonAs}{\mathbf{BiHomNonAs}}
\newcommand{\BiHomdidend}{\mathbf{BiHomDend}}

\newcommand{\BiHomtridend}{\mathbf{BiHomTridend}}
\newcommand{\BiHomRB}{\mathbf{BiHomRB}}

\newcommand{\cat}[1]{{\EuScript #1}}

\newcommand{\cO}{\cat{O}}

\newcommand{\cF}{\cat{F}}

\newcommand{\bD}{\mathbf{D}}

\begin{document}

\title{Rota-Baxter operators on BiHom-associative algebras 
and related structures
}
\author{Ling Liu \\
%EndAName
College of Mathematics, Physics and Information Engineering,\\
Zhejiang Normal University, \\
Jinhua 321004, China \\
e-mail: ntliulin@zjnu.cn \and Abdenacer Makhlouf \\
%EndAName
Universit\'{e} de Haute Alsace, \\
Laboratoire de Math\'{e}matiques, Informatique et Applications, \\
4, rue des fr\`{e}res Lumi\`{e}re, F-68093 Mulhouse, France\\
e-mail: Abdenacer.Makhlouf@uha.fr \and Claudia Menini \\
%EndAName
University of Ferrara,\\
Department of Mathematics, \\
Via Machiavelli 35, Ferrara, I-44121, Italy \\
e-mail: men@unife.it \and Florin Panaite \\
%EndAName
Institute of Mathematics of the Romanian Academy,\\
PO-Box 1-764, RO-014700 Bucharest, Romania\\
e-mail: florin.panaite@imar.ro }
\maketitle

\begin{abstract}
The purpose of this paper is to study Rota-Baxter operators for BiHom-associative algebras. 
Moreover, we introduce and discuss the properties of the notions of 
BiHom-(tri)dendriform algebra, BiHom-Zinbiel algebra and BiHom-quadri-algebra. We construct the free Rota-Baxter BiHom-associative algebra 
and present some observations about categories and functors related to Rota-Baxter structures.\\
\textbf{Keywords}: Rota-Baxter operator; BiHom-associative algebra; BiHom-dendriform algebra
\end{abstract}

%%%%%%%%%%%%%%%%%%%%%%%%%%%%%%
\section*{Introduction}
%%%%%%%%%%%%%%%%%%%%%%%%%%%%%%
The first instance of Hom-type algebras appeared in the Physics literature when looking for quantum deformations 
of some algebras of vector fields, like Witt and Virasoro algebras, in connection with oscillator algebras
(\cite{AizawaSaito,Hu}).  
A quantum deformation consists of replacing the usual derivation by a $\sigma$-derivation. 
It turns out that the algebras obtained in this way no longer satisfy the Jacobi identity but they satisfy a 
modified version involving a homomorphism. These algebras
were called Hom-Lie algebras and studied by Hartwig, Larsson and Silvestrov in \cite{JDS,DS}.
The Hom-associative algebras play the role of associative algebras in the Hom-Lie
setting. They were introduced  in \cite{ms}, where it is shown
that the commutator bracket defined by the multiplication in a Hom-associative
algebra leads naturally to a Hom-Lie algebra. The adjoint functor from the category of Hom-Lie algebras to the category of Hom-associative algebras  and the enveloping algebra were constructed in \cite{YauEnv}. A categorical approach to Hom-type algebras was considered  in \cite{stef}.  A generalization  has been given in \cite{gmmp}, where a construction of a Hom-category including a group action led to concepts of BiHom-type algebras. Hence, BiHom-associative algebras and BiHom-Lie algebras, involving two linear maps (called structure maps), were introduced. The main axioms for these types of algebras 
(BiHom-associativity, BiHom-skew-symmetry and BiHom-Jacobi condition) were dictated by categorical considerations.  

Rota-Baxter operators   first appeared in G. Baxter's work in probability and the study of fluctuation theory (\cite{Baxter}).  Afterwards, Rota-Baxter algebras were intensively studied by G. C. Rota (\cite{Rota,Rota2}) in connection with combinatorics. 
Rota-Baxter operators have appeared in a wide range of areas in pure and applied mathematics, for example under the name  \emph{multiplicativity constraint}   in the work of A. Connes and D. Kreimer (\cite{Connes-Kreimer}) about their Hopf algebra approach to renormalization  of quantum field theory. This seminal work was followed by an important development of the theory of 
Rota-Baxter algebras and their connections to other algebraic structures, see for example   \cite{agui,Bai O-operators,Bai O-operators2,KEF1,KEF-Guo1,KEF-Guo_2007,KEF-Manchon2009,
KEF-Bondia-Patras,Guo}. 
In the context of Lie algebras, Rota-Baxter operators were introduced
independently by Belavin and Drinfeld (\cite{Drinfeld}) and Semenov-Tian-Shansky (\cite{semenov}), in connection with solutions of the (modified) classical Yang-Baxter equation. Motivated by $K$-theory, Loday (\cite{loday}) introduced 
 dendriform algebras, that dichotomize an associative multiplication. It turns out that  dendriform algebras are connected to several areas in mathematics and physics. Moreover,  Rota-Baxter algebras are related to dendriform algebras via a pair of adjoint functors (\cite{KEF1,KEF-Guo2}).  

Rota-Baxter operators on associative or Lie algebras have been intensively studied. Recently, Rota-Baxter 
operators on other types of algebras started to be investigated. For instance, Rota-Baxter operators on flexible, alternative, 
Leibniz and Malcev algebras have been 
studied   in \cite{RF} and \cite{madariaga}, while Rota-Baxter operators on Hom-type algebras have been studied in \cite{makhloufrotabaxter,makhloufyau}. We begin here the study of Rota-Baxter operators on BiHom-associative algebras. 
The main question we wanted to answer is the following. Let $(A, \cdot , \alpha , \beta )$ be a BiHom-associative algebra,  $R:A\rightarrow A$ a Rota-Baxter operator of weight $\lambda $ and 
define a new multiplication on $A$ by $a*b=R(a)\cdot b+a\cdot R(b)+\lambda a\cdot b$, for all $a, b\in A$; then under 
what circumstances is $(A, *, \alpha , \beta )$ a BiHom-associative algebra? As we will see, a sufficient condition is that $R$ commutes with $\alpha $ and $\beta $. 
However, along the way, we will make a detailed study of various structures that may be related to Rota-Baxter 
operators on BiHom-associative algebras. 

The paper is organized as follows. 
In Section \ref{sec1},  we recall the definitions of Rota-Baxter operators, different types of  algebras (dendriform, Zinbiel,  tridendriform, quadri-algebra) and the notion of BiHom-associative algebra. 
Moreover, we describe the construction called the Yau twist. In Section \ref{sec2}, we introduce the notions of 
BiHom-(tri)dendriform algebra, BiHom-Zinbiel algebra and BiHom-quadri-algebra. We discuss their properties,  Yau twists and some 
basic constructions. 
In Section \ref{sec4}, we deal with Rota-Baxter structures for BiHom-associative algebras. We provide several concrete examples 
(obtained by a computer algebra system) of Rota-Baxter operators on some BiHom-associative algebras, 
then we discuss some general properties and describe some key constructions. We construct, in Section \ref{sec6}, the free Rota-Baxter BiHom-associative algebra and give some observations about categories and functors related to Rota-Baxter structures.  Finally, a connection between Rota-Baxter operators on 
BiHom-associative algebras 
and so-called weak BiHom-pseudotwistors  is provided in Section \ref{sec8}. 
%%%%%%%%%%%%%%%%%%%%%%%%%%%%%%%%%%
\section{Preliminaries}\label{sec1}
%%%%%%%%%%%%%%%%%%%%%%%%%%%%%%%%%%
\setcounter{equation}{0} %%%%%%%%%%%%%%%%%%%%%%%%%%%%

We work over a base field $\Bbbk $. All
algebras, linear spaces etc. will be over $\Bbbk $; unadorned $\otimes $
means $\otimes_{\Bbbk}$. We denote by $_{\Bbbk }\mathcal{M}$ the category of linear spaces over $\Bbbk $. 
Unless otherwise specified, the
algebras (associative or not) that will appear in what follows are
\emph{not} supposed to be unital, and a multiplication $\mu :A\otimes
A\rightarrow A$ on a linear space $A$ is denoted by juxtaposition: $\mu
(v\otimes v^{\prime })=vv^{\prime }$. For the composition of two maps $f$
and $g$, we write either $g\circ f$ or simply $gf$. For the identity
map on a linear space $A$ we use the notation $id_A$.
\begin{definition}\label{RB-DEF}
Let $A$ be a linear space and $\mu :A\ot A\rightarrow A$, $\mu (x\ot y)=xy$, for all $x, y\in A$, a
linear multiplication on $A$ and let $\lambda \in \Bbbk $. A Rota-Baxter operator of weight $\lambda $
for $(A, \mu )$ is a linear map $R:A\rightarrow A$ satisfying the so-called Rota-Baxter condition
\begin{eqnarray}
&&R(x)R(y)=R(R(x)y+xR(y)+\lambda xy), \;\;\;\forall \; x, y\in A. \label{RBrel}
\end{eqnarray}
\end{definition}

In this case, if we define on $A$ a new multiplication by $x*y=xR(y)+R(x)y+\lambda xy$,
for all $x, y\in A$, then
$R(x*y)=R(x)R(y)$, for all $x, y\in A$, and $R$ is a Rota-Baxter operator of weight $\lambda $ for
$(A, *)$. If $(A, \mu )$ is associative then $(A, *)$ is also associative.
\begin{definition} (\cite{loday}) A dendriform algebra is a triple $(A, \prec , \succ )$
consisting of a linear space $A$ and two linear operations $\prec , \succ :A\otimes A\rightarrow A$ satisfying
the conditions
(for all $x, y, z\in A$):
\begin{eqnarray}
&&(x\prec y)\prec z=x\prec (y\prec z+y\succ z),  \label{dend1} \\
&&(x \succ y)\prec z=x\succ (y\prec z), \label{dend2} \\
&&x\succ (y\succ z)=(x\prec y+x\succ y)\succ z. \label{dend3}
\end{eqnarray}

A morphism $f:(A, \prec , \succ )\rightarrow (A', \prec ', \succ ')$ of dendriform algebras is a linear map
$f:A\rightarrow A'$ satisfying $f(x\prec y)=f(x)\prec ' f(y)$ and $f(x\succ y)=f(x)\succ ' f(y)$, for all $x, y\in A$. 

A dendriform algebra $(A, \prec , \succ )$ is called commutative if $x\prec y=y\succ x$, for all $x, y\in A$. 
\end{definition}
\begin{definition} (\cite{loday})
A (right) Zinbiel algebra is an algebra $(A, \mu )$ for which
\begin{eqnarray}
&&(xy)z=x(yz)+x(zy), \;\;\; \forall \;x, y, z\in A. \label{zinb}
\end{eqnarray}
\end{definition}

By \cite{loday}, Zinbiel algebras are exactly commutative dendriform algebras, in the following sense. If $(A, \mu )$ 
is a Zinbiel algebra and we define $x\prec y=xy$ and $x\succ y=yx$, for $x, y\in A$, then $(A, \prec , \succ )$ 
is a commutative dendriform algebra. Conversely, if $(A, \prec , \succ )$ is a commutative dendriform algebra and 
we define $\mu (x\otimes y)=x\prec y$, for $x, y\in A$, then $(A, \mu )$ is a Zinbiel algebra. 
\begin{definition} (\cite{lodayronco}) A tridendriform algebra is a 4-tuple
$(A, \prec , \succ , \cdot )$, where $A$ is a linear space and $\prec , \succ , \cdot :A\otimes A\rightarrow A$ are
linear operations satisfying the conditions
(for all $x, y, z\in A$):
\begin{eqnarray}
&&(x\prec y)\prec z=x\prec (y\prec z+y\succ z+ y\cdot z),  \label{tridend1} \\
&&(x \succ y)\prec z=x\succ (y\prec z), \label{tridend2} \\
&&x\succ (y\succ z)=(x\prec y+x\succ y+x\cdot y)\succ z, \label{tridend3} \\
&&x\cdot (y\succ z)=(x\prec y)\cdot z, \label{tridend4} \\
&&x\succ (y\cdot z)=(x\succ y)\cdot z, \label{tridend5} \\
&&x\cdot (y\prec z)=(x\cdot y)\prec z, \label{tridend6} \\
&&x\cdot (y\cdot z)=(x\cdot y)\cdot z. \label{tridend7}
\end{eqnarray}

A morphism $f:(A, \prec , \succ , \cdot )\rightarrow (A', \prec ', \succ ', \cdot ')$ of tridendriform algebras is a linear map
$f:A\rightarrow A'$ satisfying $f(x\prec y)=f(x)\prec ' f(y)$, $f(x\succ y)=f(x)\succ ' f(y)$ and
$f(x\cdot y)=f(x)\cdot ' f(y)$,
for all $x, y\in A$.
\end{definition}

\begin{definition}(\cite{agui}) A quadri-algebra is a 5-tuple $(Q, \nwarrow, \swarrow, \nearrow, \searrow)$
consisting of a linear space $Q$ and four  linear maps $\nwarrow,
\swarrow, \nearrow, \searrow: Q\otimes Q\rightarrow Q$ satisfying
the axioms below (\ref{1.19})-(\ref{1.23}) (for all $x, y, z\in Q$).
In order to state them, consider the following operations:
\begin{eqnarray}
&& x\succ y:=x\nearrow y+ x\searrow y    \label{qua1} \\
&& x\prec y:=x\nwarrow y+ x\swarrow y, \label{qua2} \\
&& x\vee y:=x\searrow y+ x\swarrow y, \label{qua3} \\
&& x\wedge y:=x\nearrow y+ x\nwarrow y, \label{qua4}
\end{eqnarray}
\begin{eqnarray}
 x\ast y:&=&x\searrow y+ x\nearrow y+ x\swarrow y+ x\nwarrow y \nonumber \\
 &=& x\succ y+ x\prec y= x\vee y+ x\wedge y. \label{qua5}
\end{eqnarray}
The axioms are
\begin{eqnarray}
&& (x\nwarrow y)\nwarrow z=x \nwarrow (y\ast z), ~ (x\nearrow
y)\nwarrow z=x\nearrow (y\prec
z), \label{1.19}\\
&&(x\wedge y)\nearrow z=x\nearrow (y\succ z), ~ 
(x\swarrow y)\nwarrow z=x\swarrow (y\wedge z), \label{1.20}\\
&&(x\searrow
y)\nwarrow z=x\searrow (y\nwarrow z), ~ (x\vee
y)\nearrow z=x\searrow (y\nearrow z), \label{1.21}\\
&& (x\prec y)\swarrow z=x\swarrow (y\vee z),~ (x\succ y)\swarrow
z=x\searrow (y\swarrow z), \label{1.22} \\
&&(x\ast y)\searrow z=x\searrow (y\searrow
z). \label{1.23}
\end{eqnarray}

A morphism $f:(Q, \nwarrow, \swarrow, \nearrow, \searrow)\rightarrow
(Q', \nwarrow', \swarrow', \nearrow', \searrow')$ of quadri-algebras
is a linear map $f:Q\rightarrow Q'$ such that 
$f(x\nearrow y)=f(x)\nearrow' f(y), f(x\searrow
y)=f(x)\searrow' f(y), f(x\nwarrow y)=f(x)\nwarrow' f(y)$ and
$f(x\swarrow y)=f(x)\swarrow ' f(y)$, for all $x, y\in Q$. As a consequence, we also have 
$f(x\succ y)=f(x)\succ ' f(y)$, $f(x\prec y)=f(x)\prec ' f(y)$, $f(x\vee y)=f(x)\vee ' f(y)$, 
$f(x\wedge y)=f(x)\wedge ' f(y)$ and $f(x\ast y)=f(x)\ast ' f(y)$, for all $x, y\in Q$. 
\end{definition}
\begin{definition} (\cite{gmmp}) 
A BiHom-associative algebra over $\Bbbk $%
\textbf{\ }is a 4-tuple $\left( A,\mu ,\alpha ,\beta \right) $, where $A$ is
a $\Bbbk $-linear space, $\alpha :A\rightarrow A$, $\beta :A\rightarrow A$
and $\mu :A\otimes A\rightarrow A$ are linear maps, with notation $\mu (x\otimes y) =xy$, for all $x, y\in A$, satisfying the following
conditions, for all $x, y, z\in A$:
\begin{gather}
\alpha \circ \beta =\beta \circ \alpha , \\
\alpha (xy) =\alpha (x)\alpha (y) \text{ and }\beta (xy)=\beta (x)\beta (y) ,\quad \text{%
(multiplicativity)}  \label{eqalfabeta} \\
\alpha (x)(yz)=(xy)\beta (z).\quad \text{%
(BiHom-associativity)}  \label{eqasso}
\end{gather}

We call $\alpha $ and $\beta $ (in this order) the structure maps
of $A$.

A morphism $f:(A, \mu _A , \alpha _A, \beta _A)\rightarrow (B, \mu _B ,
\alpha _B, \beta _B)$ of BiHom-associative algebras is a linear map $%
f:A\rightarrow B$ such that $\alpha _B\circ f=f\circ \alpha _A$, $\beta
_B\circ f=f\circ \beta _A$ and $f\circ \mu_A=\mu _B\circ (f\otimes f)$.
\end{definition}

If $(A, \mu )$ is an associative algebra, where $\mu :A\otimes A\rightarrow A$ is the multiplication,  and $\alpha , \beta :A\rightarrow A$ are commuting algebra endomorphisms, then $\left( A,\mu
\circ \left( \alpha \otimes \beta \right) ,\alpha ,\beta \right) $ is a
BiHom-associative algebra, denoted by $A_{(\alpha ,\beta )}$ and called the
Yau twist of $(A, \mu )$.

More generally, let $(D, \mu , \tilde{\alpha }, \tilde{\beta })$ be a
BiHom-associative algebra and $\alpha , \beta :D\rightarrow D$ two
multiplicative linear maps such that any two of the maps $\tilde{\alpha },
\tilde{\beta }, \alpha , \beta $ commute. Then $(D, \mu \circ (\alpha
\otimes \beta ), \tilde{\alpha }\circ \alpha , \tilde{\beta }\circ \beta )$
is also a BiHom-associative algebra, denoted by $D_{(\alpha , \beta )}$.
\begin{definition} 
Let $(A, \mu _A , \alpha _A, \beta _A)$ be a BiHom-associative algebra and $(M, \alpha _M, \beta _M)$ a triple where $M$
is a linear space and $\alpha _M, \beta _M:M \rightarrow M$ are commuting linear maps.\\
(i) $(M, \alpha _M, \beta _M)$ is a left $A$-module if 
we have a linear map $A\otimes M\rightarrow M$, $a\otimes m\mapsto a\cdot m$, such that
$\alpha _M(a\cdot m)=\alpha _A(a)\cdot \alpha _M(m)$,
$\beta _M(a\cdot m)=\beta _A(a)\cdot \beta _M(m)$ and
\begin{eqnarray}
&&\alpha _A(a)\cdot (a^{\prime }\cdot m)=(a\cdot a^{\prime })\cdot \beta _M(m),\;\;\;\;\forall \;\;a, a'\in A, \;m\in M.
\label{lmod4}
\end{eqnarray}
(ii) $(M, \alpha _M, \beta _M)$ is a right $A$-module if 
we have a linear map $M\otimes A\rightarrow M$, $m\otimes a\mapsto m\cdot a$, such that
$\alpha _M(m\cdot a)=\alpha _M(m)\cdot \alpha _A(a)$,
$\beta _M(m\cdot a)=\beta _M(m)\cdot \beta _A(a)$ and
\begin{eqnarray}
&&\alpha _M(m)\cdot (a\cdot a')=(m\cdot a)\cdot \beta _A(a'), \;\;\;\;\forall \;\;a, a'\in A, \;m\in M.
\label{rmod4}
\end{eqnarray}
(iii) If $(M, \alpha _M, \beta _M)$ is 
a left and right $A$-module, 
then $M$ is called an $A$-bimodule if 
\begin{eqnarray}
&&\alpha _A(a)\cdot (m\cdot a')=(a\cdot m)\cdot \beta _A(a'), \;\;\;\;\forall \;\;a, a'\in A, \;m\in M. \label{BHbim}
\end{eqnarray}
\end{definition}

If $(A, \mu , \alpha , \beta )$ is a BiHom-associative algebra, then $(A,
\alpha , \beta )$ is an $A$-bimodule, with actions defined by $a\cdot m=am$ and $m\cdot a=ma$,
for all $a, m\in A$.

Similarly to the classical (associative) case and to the Hom-case proved in \cite{makpan}, 
one can characterize bimodules in terms of split null extensions. 
%%%%%%%%%%%%%%%%%%%%%%%%%%%%%%%%%%%%%%
\section{BiHom-(tri)dendriform algebras and BiHom-quadri-algebras}\label{sec2}
%%%%%%%%%%%%%%%%%%%%%%%%%%%%%
\setcounter{equation}{0}
%%%%%%%%%%%%%%%%%%%%%%%%%%%%
In this section, we introduce the notions of BiHom-dendriform algebra, BiHom-tridendriform algebra and BiHom-quadri-algebra, generalizing the Hom-type structures given in \cite{an,makhloufrotabaxter,makhloufyau}. Moreover, we provide some key constructions.
\begin{definition} A BiHom-dendriform algebra is a 5-tuple $(A, \prec , \succ , \alpha , \beta )$
consisting of a linear space $A$ and linear maps $\prec , \succ :A\otimes A\rightarrow A$ and
$\alpha , \beta :A\rightarrow A$
satisfying
the conditions
\begin{eqnarray}
&&\alpha \circ \beta =\beta \circ \alpha , \label{BiHomdend1} \\
&&\alpha (x\prec y)=\alpha (x)\prec \alpha (y), ~ 
\alpha (x\succ y)=\alpha (x)\succ \alpha (y), \label{BiHomdend3} \\
&&\beta (x\prec y)=\beta (x)\prec \beta (y), ~
\beta (x\succ y)=\beta (x)\succ \beta (y), \label{BiHomdend5} \\
&&(x\prec y)\prec \beta (z)=\alpha (x)\prec (y\prec z+y\succ z),  \label{BiHomdend6} \\
&&(x \succ y)\prec \beta (z)=\alpha (x)\succ (y\prec z), \label{BiHomdend7} \\
&&\alpha (x)\succ (y\succ z)=(x\prec y+x\succ y)\succ \beta (z),  \label{BiHomdend8}
\end{eqnarray}
for all $x, y, z\in A$. 
We call $\alpha $ and $\beta $ (in this order) the structure maps
of $A$.

A morphism $f:(A, \prec , \succ , \alpha , \beta )\rightarrow (A', \prec ', \succ ', \alpha ', \beta ')$ of
BiHom-dendriform algebras is a linear map
$f:A\rightarrow A'$ satisfying $f(x\prec y)=f(x)\prec ' f(y)$ and $f(x\succ y)=f(x)\succ ' f(y)$, for all $x, y\in A$,
as well as $f\circ \alpha =\alpha '\circ f$ and $f\circ \beta =\beta '\circ f$.

A BiHom-dendriform algebra $(A, \prec , \succ , \alpha , \beta )$ is called commutative if 
\begin{eqnarray}
\beta (x)\prec \alpha (y)=\beta (y)\succ \alpha (x),\;\;\; \forall \; x, y\in A. \label{comBHdendri}
\end{eqnarray}
\end{definition}
\begin{proposition} \label{Yaudend}
Let $(A, \prec , \succ )$ be a dendriform algebra and $\alpha , \beta :A\rightarrow A$ two
commuting dendriform algebra endomorphisms. Define $\prec _{(\alpha , \beta )}, \succ _{(\alpha , \beta )}:
A\ot A\rightarrow A$ by
\begin{eqnarray*}
&&x\prec _{(\alpha , \beta )}y=\alpha (x)\prec \beta (y) \;\;\;\;\;and\;\;\;\;\;
x\succ _{(\alpha , \beta )}y=\alpha (x)\succ \beta (y),
\end{eqnarray*}
for all $x, y\in A$. Then $A_{(\alpha , \beta )}:=(A, \prec _{(\alpha , \beta )}, \succ _{(\alpha , \beta )},
\alpha , \beta )$ is a BiHom-dendriform algebra, called the Yau twist of $A$. Moreover, assume that
$(A', \prec ', \succ ')$ is another dendriform algebra and $\alpha ', \beta ':A'\rightarrow A'$ are
two commuting dendriform algebra endomorphisms and $f:A\rightarrow A'$ is a morphism of
dendriform algebras satisfying $f\circ \alpha =\alpha '\circ f$ and $f\circ \beta =\beta '\circ f$. Then
$f:A_{(\alpha , \beta )}\rightarrow A'_{(\alpha ', \beta ')}$ is a morphism of BiHom-dendriform algebras.
\end{proposition}
\begin{proof}
We only prove (\ref{BiHomdend6}) and leave the rest to the reader. By using the
formulae for $\prec _{(\alpha , \beta )}$ and $\succ _{(\alpha , \beta )}$ and the fact that $\alpha $ and
$\beta $ are commuting dendriform algebra endomorphisms, one can compute, for all $x, y, z\in A$:
\begin{eqnarray*}
&&(x\prec _{(\alpha , \beta )}y)\prec _{(\alpha , \beta )}\beta (z)=(\alpha ^2(x)\prec \alpha \beta (y))\prec \beta ^2(z), \\
&&\alpha (x)\prec _{(\alpha , \beta )}(y\prec _{(\alpha , \beta )}z)=\alpha ^2(x)\prec (\alpha \beta (y)\prec \beta ^2(z)), \\
&&\alpha (x)\prec _{(\alpha , \beta )}(y\succ _{(\alpha , \beta )}z)=\alpha ^2(x)\prec (\alpha \beta (y)\succ \beta ^2(z)). 
\end{eqnarray*}
Thus, (\ref{BiHomdend6}) follows from (\ref{dend1}) applied to the elements
$\alpha ^2(x)$, $\alpha \beta (y)$, $\beta ^2(z)$.
\end{proof}
\begin{remark}
More generally,
let $(A, \prec , \succ , \alpha , \beta )$ be a BiHom-dendriform algebra and $\tilde{\alpha }, \tilde{\beta }:
A\rightarrow A$ two morphisms of BiHom-dendriform algebras such that any two of the maps $\alpha, \beta ,
\tilde{\alpha }, \tilde{\beta }$ commute. Define new multiplications on $A$ by
$x\prec 'y=\tilde{\alpha }(x)\prec \tilde{\beta }(y)$ and 
$x\succ 'y=\tilde{\alpha }(x)\succ \tilde{\beta }(y)$,
for all $x, y\in A$. Then one can prove that $(A, \prec ', \succ ', \alpha \circ \tilde{\alpha }, \beta \circ \tilde{\beta })$
is a BiHom-dendriform algebra.
\end{remark}

We introduce now the BiHom-analogue of Zinbiel algebras. 
\begin{definition}
A BiHom-Zinbiel algebra is a 4-tuple $(A, \cdot , \alpha , \beta )$ where $(A, \cdot )$ is an algebra, $\alpha , \beta :
A\rightarrow A$ are commuting algebra maps and the following relations are satisfied, for all $x, y, z\in A$:
\begin{eqnarray}
&&(x\cdot \beta (y))\cdot \alpha \beta (z)=\alpha (x)\cdot (\beta (y)\cdot \alpha (z))+\alpha (x)\cdot (\beta (z)\cdot 
\alpha (y)), \label{BHzinb1} \\
&&(\beta (y)\cdot \alpha (x))\cdot \alpha \beta (z)=(\beta (y)\cdot \beta (z))\cdot \alpha ^2(x). \label{BHzinb2}
\end{eqnarray}
\end{definition}

We can obtain examples of BiHom-Zinbiel algebras by twisting Zinbiel algebras:  
\begin{proposition}
Let $(A, \mu )$ be a Zinbiel algebra and $\alpha , \beta :A\rightarrow A$ two commuting algebra maps. Define a new 
multiplication on $A$ by $x\cdot y=\alpha (x)\beta (y)$, for all $x, y\in A$. Then 
$(A, \cdot , \alpha , \beta )$ is a BiHom-Zinbiel algebra, called the Yau twist of $A$. 
\end{proposition}
\begin{proof}
For $x, y, z\in A$ we compute: 
\begin{eqnarray*}
(x\cdot \beta (y))\cdot \alpha \beta (z)&=&(\alpha ^2(x)\alpha \beta ^2(y))\alpha \beta ^2(z)\\
&\overset{(\ref{zinb})}{=}&\alpha ^2(x)(\alpha \beta ^2(y)\alpha \beta ^2(z))+
\alpha ^2(x)(\alpha \beta ^2(z)\alpha \beta ^2(y))\\
&=&\alpha ^2(x)\beta (\alpha \beta (y)\alpha \beta (z))+
\alpha ^2(x)\beta (\alpha \beta (z)\alpha \beta (y))\\
&=&\alpha ^2(x)\beta (\beta (y)\cdot \alpha (z))+
\alpha ^2(x)\beta (\beta (z)\cdot \alpha (y))\\
&=&\alpha (x)\cdot (\beta (y)\cdot \alpha (z))+
\alpha (x)\cdot (\beta (z)\cdot \alpha (y)), 
\end{eqnarray*}
\begin{eqnarray*}
(\beta (y)\cdot \alpha (x))\cdot \alpha \beta (z)&=&(\alpha ^2\beta (y)\alpha ^2\beta (x))\alpha \beta ^2(z)\\
&\overset{(\ref{zinb})}{=}&(\alpha ^2\beta (y)\alpha \beta ^2(z))\alpha ^2\beta (x)\\
&=&\alpha (\alpha \beta (y)\beta ^2(z))\beta (\alpha ^2(x))=(\beta (y)\cdot \beta (z))\cdot \alpha ^2(x), 
\end{eqnarray*}
finishing the proof. 
\end{proof}
\begin{remark}
An immediate consequence of (\ref{BHzinb1}) is that 
\begin{eqnarray}
&&(x\cdot \beta (y))\cdot \alpha \beta (z)=(x\cdot \beta (z))\cdot \alpha \beta (y), \;\;\;\forall \; x, y, z\in A. 
\label{consBHzinb1}
\end{eqnarray}
\end{remark}
\begin{remark}
In a BiHom-Zinbiel algebra $(A, \cdot , \alpha , \beta )$ with bijective $\alpha $ and $\beta $, the relation 
(\ref{BHzinb2}) is superfluous (being a consequence of (\ref{BHzinb1})). Indeed, we can compute: 
\begin{eqnarray*}
(\beta (y)\cdot \alpha (x))\cdot \alpha \beta (z)&=&(\beta (y)\cdot \beta (\alpha \beta ^{-1}(x)))\cdot \alpha \beta (z)\\
&\overset{(\ref{consBHzinb1})}{=}&(\beta (y)\cdot \beta (z))\cdot \alpha \beta (\alpha \beta ^{-1}(x))
=(\beta (y)\cdot \beta (z))\cdot \alpha ^2(x), \;\;\;q.e.d.
\end{eqnarray*}
\end{remark}
\begin{proposition}
Let $(A, \prec , \succ , \alpha , \beta )$ be a commutative BiHom-dendriform algebra. Define the operation $x\cdot y=x\prec y$, 
for all $x, y\in A$. Then $(A, \cdot , \alpha , \beta )$ is a BiHom-Zinbiel algebra. 
\end{proposition}
\begin{proof}
We compute: 
\begin{eqnarray*}
(x\cdot \beta (y))\cdot \alpha \beta (z)&=&(x\prec \beta (y))\prec \alpha \beta (z)\\
&\overset{(\ref{BiHomdend6})}{=}&\alpha (x)\prec (\beta (y)\prec \alpha (z))+\alpha (x)\prec (\beta (y)\succ \alpha (z))\\
&\overset{(\ref{comBHdendri})}{=}&\alpha (x)\prec (\beta (y)\prec \alpha (z))+\alpha (x)\prec (\beta (z)\prec \alpha (y))\\
&=&\alpha (x)\cdot (\beta (y)\cdot \alpha (z))+\alpha (x)\cdot (\beta (z)\cdot \alpha (y)), 
\end{eqnarray*}
\begin{eqnarray*}
(\beta (y)\cdot \alpha (x))\cdot \alpha \beta (z)&=&(\beta (y)\prec \alpha (x))\prec \alpha \beta (z)\\
&\overset{(\ref{comBHdendri})}{=}&(\beta (x)\succ \alpha (y))\prec \alpha \beta (z)\\
&\overset{(\ref{BiHomdend7})}{=}&\alpha \beta (x)\succ (\alpha (y)\prec \alpha (z))=
\beta (\alpha (x))\succ \alpha (y\prec z)\\
&\overset{(\ref{comBHdendri})}{=}&\beta (y\prec z)\prec \alpha ^2(x)=(\beta (y)\cdot \beta (z))\cdot \alpha ^2(x), 
\end{eqnarray*}
finishing the proof. 
\end{proof}

The converse holds in the case of bijective structure maps. 
\begin{proposition}
Let $(A, \cdot , \alpha , \beta )$ be a BiHom-Zinbiel algebra such that $\alpha $ and $\beta $ are bijective. Define new operations 
on $A$ by $x\prec y=x\cdot y$ and $x\succ y=\beta \alpha ^{-1}(y)\cdot \alpha \beta ^{-1}(x)$, for all 
$x, y\in A$. Then $(A, \prec , \succ , \alpha , \beta )$ is a commutative BiHom-dendriform algebra. 
\end{proposition}
\begin{proof}
Obviously, $A$ is commutative since 
$\beta (x)\succ \alpha (y)=\beta (y)\cdot \alpha (x)=\beta (y)\prec \alpha (x)$. Clearly, $\alpha $ and $\beta $ are 
multiplicative with respect to $\prec $ and $\succ $. Now we prove (\ref{BiHomdend6}), (\ref{BiHomdend7}) and 
(\ref{BiHomdend8}): 
\begin{eqnarray*}
(x\prec y)\prec \beta (z)&=&(x\cdot y)\cdot \beta (z)=(x\cdot \beta (\beta ^{-1}(y)))\cdot \alpha \beta (\alpha ^{-1}(z))\\
&\overset{(\ref{BHzinb1})}{=}&\alpha (x)\cdot (y\cdot z)+\alpha (x)\cdot (\beta \alpha ^{-1}(z)\cdot \alpha \beta ^{-1}(y))\\
&=&\alpha (x)\prec (y\prec z)+\alpha (x)\prec (y\succ z), 
\end{eqnarray*}
\begin{eqnarray*}
(x\succ y)\prec \beta (z)&=&(\beta \alpha ^{-1}(y)\cdot \alpha \beta ^{-1}(x))\cdot \beta (z)\\
&=&(\beta (\alpha ^{-1}(y))\cdot \alpha (\beta ^{-1}(x)))\cdot \alpha \beta (\alpha ^{-1}(z))\\
&\overset{(\ref{BHzinb2})}{=}&(\beta \alpha ^{-1}(y)\cdot \beta \alpha ^{-1}(z))\cdot \alpha ^2\beta ^{-1}(x)\\
&=&\beta \alpha ^{-1}(y\prec z)\cdot \alpha \beta ^{-1}(\alpha (x))=\alpha (x)\succ (y\prec z), 
\end{eqnarray*}
\begin{eqnarray*}
\alpha (x)\succ (y\succ z)&=&\alpha (x)\succ (\beta \alpha ^{-1}(z)\cdot \alpha \beta ^{-1}(y))=
(\beta ^2\alpha ^{-2}(z)\cdot y)\cdot \alpha ^2\beta ^{-1}(x)\\
&=&(\beta ^2\alpha ^{-2}(z)\cdot \beta (\beta ^{-1}(y)))\cdot \alpha \beta (\alpha \beta ^{-2}(x))\\
&\overset{(\ref{BHzinb1})}{=}&\beta ^2\alpha ^{-1}(z)\cdot (y\cdot \alpha ^2\beta ^{-2}(x))+
\beta ^2\alpha ^{-1}(z)\cdot (\alpha \beta ^{-1}(x)\cdot \alpha \beta ^{-1}(y))\\
&=&\beta ^2\alpha ^{-1}(z)\cdot (\beta \alpha ^{-1}(\alpha \beta ^{-1}(y))\cdot \alpha \beta ^{-1}(\alpha \beta ^{-1}(x)))\\
&&+\beta ^2\alpha ^{-1}(z)\cdot \alpha \beta ^{-1}(x\prec y)\\
&=&\beta ^2\alpha ^{-1}(z)\cdot (\alpha \beta ^{-1}(x)\succ \alpha \beta ^{-1}(y))+
\beta ^2\alpha ^{-1}(z)\cdot \alpha \beta ^{-1}(x\prec y)\\
&=&\beta ^2\alpha ^{-1}(z)\cdot \alpha \beta ^{-1}(x\succ y)+ 
\beta ^2\alpha ^{-1}(z)\cdot \alpha \beta ^{-1}(x\prec y)\\
&=&(x\succ y)\succ \beta (z)+(x\prec y)\succ \beta (z),
\end{eqnarray*}
finishing the proof.
\end{proof}
\begin{definition} A BiHom-tridendriform algebra is a 6-tuple $(A, \prec , \succ , \cdot , \alpha , \beta )$,
where $A$ is a linear space and $\prec , \succ , \cdot :A\otimes A\rightarrow A$ and
$\alpha , \beta :A\rightarrow A$ are linear maps
satisfying
\begin{eqnarray}
&&\alpha \circ \beta =\beta \circ \alpha , \label{BiHomtridend1} \\
&&\alpha (x\prec y)=\alpha (x)\prec \alpha (y), ~
\alpha (x\succ y)=\alpha (x)\succ \alpha (y), ~
\alpha (x\cdot y)=\alpha (x)\cdot \alpha (y), \label{BiHomtridend4} \\
&&\beta (x\prec y)=\beta (x)\prec \beta (y), ~
\beta (x\succ y)=\beta (x)\succ \beta (y), ~
\beta (x\cdot y)=\beta (x)\cdot \beta (y), \label{BiHomtridend7} \\
&&(x\prec y)\prec \beta (z)=\alpha (x)\prec (y\prec z+y\succ z+y\cdot z),  \label{BiHomtridend8} \\
&&(x \succ y)\prec \beta (z)=\alpha (x)\succ (y\prec z), \label{BiHomtridend9} \\
&&\alpha (x)\succ (y\succ z)=(x\prec y+x\succ y+x\cdot y)\succ \beta (z), \label{BiHomtridend10} \\
&&\alpha (x)\cdot (y\succ z)=(x\prec y)\cdot \beta (z), \label{BiHomtridend11} \\
&&\alpha (x)\succ (y\cdot z)=(x\succ y)\cdot \beta (z), \label{BiHomtridend12} \\
&&\alpha (x)\cdot (y\prec z)=(x\cdot y)\prec \beta (z), \label{BiHomtridend13} \\
&&\alpha (x)\cdot (y\cdot z)=(x\cdot y)\cdot \beta (z),  \label{BiHomtridend14}
\end{eqnarray}
for all $x, y, z\in A$. 
We call $\alpha $ and $\beta $ (in this order) the structure maps
of $A$.

A morphism $f:(A, \prec , \succ , \cdot , \alpha , \beta )\rightarrow (A', \prec ', \succ ', \cdot ', \alpha ', \beta ')$ of
BiHom-tridendriform algebras is a linear map
$f:A\rightarrow A'$ satisfying $f(x\prec y)=f(x)\prec ' f(y)$, $f(x\succ y)=f(x)\succ ' f(y)$ and
$f(x\cdot y)=f(x)\cdot ' f(y)$,
for all $x, y\in A$,
as well as $f\circ \alpha =\alpha '\circ f$ and $f\circ \beta =\beta '\circ f$.
\end{definition}

A BiHom-dendriform algebra $(A, \prec, \succ , \alpha , \beta )$ is also a BiHom-tridendriform algebra,
with the same operations $\prec $, $\succ $ and the same maps $\alpha $, $\beta $ and with the operation
$\cdot $ defined by $x\cdot y=0$ for all $x, y\in A$. Conversely, we have:
\begin{proposition} \label{DTDD}
Let $(A, \prec , \succ , \cdot , \alpha , \beta )$ be a BiHom-tridendriform algebra. Then
$(A, \prec ', \succ ', \alpha , \beta )$ is a BiHom-dendriform algebra, where the new operations $\prec '$ and
$\succ '$ are defined by $x\prec 'y=x\prec y+x\cdot y$ and $x\succ 'y=x\succ y$, for all $x, y\in A$.
\end{proposition}
\begin{proof}
The relations (\ref{BiHomdend1})-(\ref{BiHomdend5}) for $\prec '$ and $\succ '$ are obvious. We check
(\ref{BiHomdend6})-(\ref{BiHomdend8}):
\begin{eqnarray*}
(x\prec 'y)\prec '\beta (z)&=&(x\prec y+x\cdot y)\prec '\beta (z)\\
&=&(x\prec y+x\cdot y)\prec \beta (z)+(x\prec y+x\cdot y)\cdot \beta (z)\\
&=&(x\prec y)\prec \beta (z)+(x\cdot y)\prec \beta (z)\\
&&+(x\prec y)\cdot \beta (z)+(x\cdot y)\cdot \beta (z)\\
&\overset{(\ref{BiHomtridend8}), \;(\ref{BiHomtridend11}),\;(\ref{BiHomtridend13}),\;(\ref{BiHomtridend14})}{=}&
\alpha (x)\prec (y\prec z+y\succ z+y\cdot z)+\alpha (x)\cdot (y\prec z)\\
&&+\alpha (x)\cdot (y\succ z)+\alpha (x)\cdot (y\cdot z)\\
&=&\alpha (x)\prec (y\prec z+y\succ z+y\cdot z)\\
&&+\alpha (x)\cdot (y\prec z+y\succ z+y\cdot z)\\
&=&\alpha (x)\prec '(y\prec z+y\succ z+y\cdot z)\\
&=&\alpha (x)\prec '(y\prec 'z+y\succ 'z),
\end{eqnarray*}
\begin{eqnarray*}
(x\succ 'y)\prec '\beta (z)&=&(x\succ y)\prec \beta (z)+(x\succ y)\cdot \beta (z)\\
&\overset{(\ref{BiHomtridend9}), \;(\ref{BiHomtridend12})}{=}&\alpha (x)\succ (y\prec z)+\alpha (x)\succ (y\cdot z)\\
&=&\alpha (x)\succ (y\prec z+y\cdot z)
=\alpha (x)\succ '(y\prec 'z),
\end{eqnarray*}
\begin{eqnarray*}
\alpha (x)\succ '(y\succ 'z)&=&\alpha (x)\succ (y\succ z)\\
&\overset{(\ref{BiHomtridend10})}{=}&(x\prec y+x\succ y+x\cdot y)\succ \beta (z)\\
&=&(x\prec 'y+x\succ 'y)\succ '\beta (z),
\end{eqnarray*}
finishing the proof.
\end{proof}
\begin{proposition}
Let $(A, \prec , \succ , \cdot )$ be a tridendriform algebra and $\alpha , \beta :A\rightarrow A$ two
commuting tridendriform algebra endomorphisms. Define $\prec _{(\alpha , \beta )}, \succ _{(\alpha , \beta )},
\cdot _{(\alpha , \beta )}:
A\ot A\rightarrow A$ by
\begin{eqnarray*}
&&x\prec _{(\alpha , \beta )}y=\alpha (x)\prec \beta (y), \;\;\;\;\;
x\succ _{(\alpha , \beta )}y=\alpha (x)\succ \beta (y), \;\;\;\;\;
x\cdot _{(\alpha , \beta )}y=\alpha (x)\cdot \beta (y),
\end{eqnarray*}
for all $x, y\in A$. Then $A_{(\alpha , \beta )}:=(A, \prec _{(\alpha , \beta )}, \succ _{(\alpha , \beta )},
\cdot _{(\alpha , \beta )},
\alpha , \beta )$ is a BiHom-tridendriform algebra, called the Yau twist of $A$. Moreover, assume that
$(A', \prec ', \succ ', \cdot ' )$ is another tridendriform algebra and $\alpha ', \beta ':A'\rightarrow A'$ are
two commuting tridendriform algebra endomorphisms and $f:A\rightarrow A'$ is a morphism of
tridendriform algebras satisfying $f\circ \alpha =\alpha '\circ f$ and $f\circ \beta =\beta '\circ f$. Then
$f:A_{(\alpha , \beta )}\rightarrow A'_{(\alpha ', \beta ')}$ is a morphism of BiHom-tridendriform algebras.
\end{proposition}
\begin{proof}
Similar to the proof of Proposition \ref{Yaudend} and left to the reader.
\end{proof}
\begin{remark}
More generally, one can prove that, if
$(A, \prec , \succ , \cdot ,\alpha , \beta )$ is a BiHom-tridendriform algebra and $\tilde{\alpha }, \tilde{\beta }:
A\rightarrow A$ are two morphisms of BiHom-tridendriform algebras such that any two of the maps $\alpha, \beta ,
\tilde{\alpha }, \tilde{\beta }$ commute, and we define new multiplications on $A$ by
$x\prec 'y=\tilde{\alpha }(x)\prec \tilde{\beta }(y)$, 
$x\succ 'y=\tilde{\alpha }(x)\succ \tilde{\beta }(y)$ and 
$x\cdot 'y=\tilde{\alpha }(x)\cdot \tilde{\beta }(y)$, 
for all $x, y\in A$, then $(A, \prec ', \succ ', \cdot ', \alpha \circ \tilde{\alpha }, \beta \circ \tilde{\beta })$
is a BiHom-tridendriform algebra.
\end{remark}
\begin{proposition} \label{BHTDBH}
Let $(A, \prec , \succ , \cdot , \alpha , \beta )$ be a BiHom-tridendriform algebra.
Define a multiplication $*:A\ot A\rightarrow A$ by
$x*y=x\prec y+x\succ y+x\cdot y$, for all $x, y\in A$.
Then $(A, *, \alpha , \beta )$ is a BiHom-associative algebra.
\end{proposition}
\begin{proof}
Obviously we have $\alpha (x*y)=\alpha (x)*\alpha (y)$ and $\beta (x*y)=\beta (x)*\beta (y)$,
for all $x, y\in A$, by using (\ref{BiHomtridend4}) and (\ref{BiHomtridend7}). Now we compute, for $x, y, z\in A$:
\begin{eqnarray*}
\alpha (x)*(y*z)&=&\alpha (x)*(y\prec z)+\alpha (x)*(y\succ z)+\alpha (x)*(y\cdot z)\\
&=&\alpha (x)\prec (y\prec z)+\alpha (x)\succ (y\prec z)+\alpha (x)\cdot (y\prec z)\\
&&+\alpha (x)\prec (y\succ z)+\alpha (x)\succ (y\succ z)+\alpha (x)\cdot (y\succ z)\\
&&+\alpha (x)\prec (y\cdot z)+\alpha (x)\succ (y\cdot z)+\alpha (x)\cdot (y\cdot z)\\
&=&\alpha (x)\prec (y\prec z+y\succ z+y\cdot z)+\alpha (x)\succ (y\prec z)\\
&&+\alpha (x)\succ (y\succ z)+\alpha (x)\cdot (y\prec z)+\alpha (x)\cdot (y\succ z)\\
&&+\alpha (x)\succ (y\cdot z)+\alpha (x)\cdot (y\cdot z)\\
&\overset{(\ref{BiHomtridend8})-(\ref{BiHomtridend10})}{=}&(x\prec y)\prec \beta (z)+
(x\succ y)\prec \beta (z)+(x\prec y)\succ \beta (z)\\
&&+(x\succ y)\succ \beta (z)+(x\cdot y)\succ \beta (z)+\alpha (x)\cdot (y\prec z)\\
&&+\alpha (x)\cdot (y\succ z)
+\alpha (x)\succ (y\cdot z)+\alpha (x)\cdot (y\cdot z)\\
&\overset{(\ref{BiHomtridend11})-(\ref{BiHomtridend14})}{=}&(x\prec y)\prec \beta (z)+
(x\succ y)\prec \beta (z)+(x\prec y)\succ \beta (z)\\
&&+(x\succ y)\succ \beta (z)+(x\cdot y)\succ \beta (z)+(x\cdot y)\prec \beta (z)\\
&&+(x\prec y)\cdot \beta (z)+(x\succ y)\cdot \beta (z)+(x\cdot y)\cdot \beta (z)\\
&=&(x\prec y+x\succ y+x\cdot y)\prec \beta (z)\\
&&+(x\prec y+x\succ y+x\cdot y)\succ \beta (z)\\
&&+(x\prec y+x\succ y+x\cdot y)\cdot \beta (z)\\
&=&(x*y)\prec \beta (z)+(x*y)\succ \beta (z)+(x*y)\cdot \beta (z)
=(x*y)*\beta (z),
\end{eqnarray*}
finishing the proof.
\end{proof}
\begin{corollary}\label{BHDBHA}
Let $(A, \prec , \succ , \alpha , \beta )$ be a BiHom-dendriform algebra. Define a multiplication $*:A\ot A\rightarrow A$ by
$x*y=x\prec y+x\succ y$. Then $(A, *, \alpha , \beta )$ is a BiHom-associative algebra.
\end{corollary}

Similarly to the characterization of dendriform algebras in terms of bimodules (see for instance \cite{nankai}, Section 6.3)  
and by using Corollary \ref{BHDBHA}, 
we obtain the following characterization of BiHom-dendriform algebras: 
\begin{proposition}\label{BHdendbimod}
Let $A$ be a linear space, $\prec , \succ :A\otimes A\rightarrow A$ linear maps and $\alpha , \beta :A\rightarrow A$ two 
commuting linear maps that are multiplicative with respect to $\prec $ and $\succ $. Define $a*b=a\prec b+a\succ b$, for all 
$a, b\in A$. Then $(A, \prec , \succ , \alpha , \beta )$ is a BiHom-dendriform algebra if and only if $(A, *, \alpha , \beta )$ 
is a BiHom-associative algebra and 
$(A, \alpha , \beta )$ is a bimodule over $(A, *, \alpha , \beta )$, 
with actions $a\cdot m=a\succ m$ and $m\cdot a=m\prec a$, for all $a, m\in A$.
\end{proposition}

We introduce now the BiHom version of quadri-algebras (for the Hom version see \cite{an}). 
\begin{definition} A BiHom-quadri-algebra is a 7-tuple $(Q, \nwarrow, \swarrow, \nearrow, \searrow, \alpha , \beta )$
consisting of a linear space $Q$ and linear maps $\nwarrow,
\swarrow, \nearrow, \searrow: Q\otimes Q\rightarrow Q$ and $\alpha ,
\beta : Q\rightarrow Q$ satisfying the axioms below
(\ref{BiHomqua6})-(\ref{BiHomqua15}) (for all $x, y, z\in Q$). To state them, consider the following operations:
\begin{eqnarray}
&& x\succ y:=x\nearrow y+ x\searrow y    \label{BiHomqua1} \\
&& x\prec y:=x\nwarrow y+ x\swarrow y, \label{BiHomqua2} \\
&& x\vee y:=x\searrow y+ x\swarrow y, \label{BiHomqua3} \\
&& x\wedge y:=x\nearrow y+ x\nwarrow y, \label{BiHomqua4}
\end{eqnarray}
\begin{eqnarray}
 x\ast y:&=&x\searrow y+ x\nearrow y+ x\swarrow y+ x\nwarrow y \nonumber \\
 &=& x\succ y+ x\prec y= x\vee y+ x\wedge y. \label{BiHomqua5}
\end{eqnarray}
The axioms are
\begin{eqnarray}
&&\alpha \circ \beta =\beta \circ \alpha , \label{BiHomqua6} \\
&&\alpha (x\nearrow y)=\alpha (x)\nearrow \alpha (y), ~~~~ \alpha (x\searrow y)=\alpha (x)\searrow \alpha (y), \label{BiHomqua7} \\
&&\alpha (x\nwarrow y)=\alpha (x)\nwarrow \alpha (y), ~~~~ \alpha (x\swarrow y)=\alpha (x)\swarrow \alpha (y), \label{BiHomqua8} \\
&&\beta (x\nearrow y)=\beta (x)\nearrow \beta (y),~~~~  \beta (x\searrow y)=\beta (x)\searrow \beta (y)\label{BiHomqua9} \\
&&\beta (x\nwarrow y)=\beta (x)\nwarrow \beta (y),~~~~ \beta (x\swarrow y)=\beta (x)\swarrow \beta (y) \label{BiHomqua10} \\
&& (x\nwarrow y)\nwarrow \beta(z)=\alpha(x) \nwarrow (y\ast z), ~~~~
(x\nearrow y)\nwarrow \beta(z)=\alpha(x)\nearrow (y\prec
z),\label{BiHomqua11}\\
&& (x\wedge y)\nearrow \beta(z)=\alpha(x)\nearrow (y\succ z),~~~~  (x\swarrow y)\nwarrow \beta(z)=\alpha(x)\swarrow (y\wedge z), \label{BiHomqua12}\\
&& (x\searrow y)\nwarrow \beta(z)=\alpha(x)\searrow (y\nwarrow
z),~~~~ (x\vee
y)\nearrow \beta(z)=\alpha(x)\searrow (y\nearrow z), \label{BiHomqua13} \\
&& (x\prec y)\swarrow \beta(z)=\alpha(x)\swarrow (y\vee z),~~~~  (x\succ y)\swarrow \beta(z)=\alpha(x)\searrow (y\swarrow z), \label{BiHomqua14}\\
&& (x\ast y)\searrow \beta(z)=\alpha(x)\searrow (y\searrow z).
\label{BiHomqua15}
\end{eqnarray}

A morphism $f:(Q, \nwarrow, \swarrow, \nearrow, \searrow, \alpha ,
\beta )\rightarrow (Q', \nwarrow', \swarrow', \nearrow', \searrow',
\alpha' , \beta' )$ of BiHom-quadri-algebras is a linear map
$f:Q\rightarrow Q'$ satisfying $f(x\nearrow y)=f(x)\nearrow' f(y),
f(x\searrow y)=f(x)\searrow' f(y), f(x\nwarrow y)=f(x)\nwarrow'
f(y)$ and $f(x\swarrow y)=f(x)\swarrow ' f(y)$, for all $x, y\in Q$,
as well as $f\circ \alpha =\alpha '\circ f$ and $f\circ \beta =\beta
'\circ f$.
\end{definition}
\begin{proposition} \label{quad}
Let $(Q, \nwarrow, \swarrow, \nearrow, \searrow )$ be a
quadri-algebra and $\alpha , \beta :Q\rightarrow Q$ two commuting
quadri-algebra endomorphisms. Define $\searrow _{(\alpha , \beta )},
\nearrow _{(\alpha , \beta )}, \swarrow _{(\alpha , \beta )},
\nwarrow _{(\alpha , \beta )}: Q\ot Q\rightarrow Q$ by
\begin{eqnarray*}
&&x\searrow _{(\alpha , \beta )}y=\alpha (x)\searrow \beta (y),
\quad \quad x\nearrow _{(\alpha , \beta )}y=\alpha (x)\nearrow \beta
(y),\\
&&x\swarrow _{(\alpha , \beta )}y=\alpha (x)\swarrow \beta (y),
\quad \quad x\nwarrow _{(\alpha , \beta )}y=\alpha (x)\nwarrow \beta
(y),
\end{eqnarray*}
for all $x, y\in Q$. 
Then $Q_{(\alpha , \beta )}:=(Q, \nwarrow _{(\alpha , \beta )},  \swarrow
_{(\alpha , \beta )}, \nearrow _{(\alpha , \beta )},
\searrow
_{(\alpha , \beta )},  \alpha , \beta )$ is a BiHom-quadri-algebra, called the Yau twist of $Q$. Moreover,
assume that $(Q', \nwarrow', \swarrow', \nearrow', \searrow' )$ is
another quadri-algebra and $\alpha ', \beta ':Q'\rightarrow Q'$ are
two commuting quadri-algebra endomorphisms and $f:Q\rightarrow Q'$
is a morphism of quadri-algebras satisfying $f\circ \alpha =\alpha
'\circ f$ and $f\circ \beta =\beta '\circ f$. Then $f:Q_{(\alpha ,
\beta )}\rightarrow Q'_{(\alpha ', \beta ')}$ is a morphism of
BiHom-quadri algebras.
\end{proposition}
\begin{proof}
We only prove (\ref{BiHomqua11}) and leave the
rest to the reader. We denote by $x\succ _{(\alpha , \beta )}y:=x\nearrow _{(\alpha , \beta )}y+ 
x\searrow _{(\alpha , \beta )}y$, 
$x\prec _{(\alpha , \beta )}y:=x\nwarrow _{(\alpha , \beta )}y+ x\swarrow _{(\alpha , \beta )}y$, 
$x\vee _{(\alpha , \beta )}y:=x\searrow _{(\alpha , \beta )}y+ x\swarrow _{(\alpha , \beta )}y$, 
$x\wedge _{(\alpha , \beta )}y:=x\nearrow _{(\alpha , \beta )}y+ x\nwarrow _{(\alpha , \beta )}y$ 
and $x\ast _{(\alpha , \beta )}y:=x\searrow _{(\alpha , \beta )}y+ x\nearrow _{(\alpha , \beta )}y+ 
x\swarrow _{(\alpha , \beta )}y+ x\nwarrow _{(\alpha , \beta )}y$, for all $x, y\in Q$. 
It is easy to get $x\succ _{(\alpha , \beta )}
y= \alpha (x)\succ \beta (y), x\prec _{(\alpha , \beta )} y= \alpha
(x)\prec \beta (y), x\vee _{(\alpha , \beta )} y= \alpha (x)\vee
\beta (y), x\wedge _{(\alpha , \beta )} y= \alpha (x)\wedge \beta
(y)$ and $x\ast _{(\alpha , \beta )} y= \alpha (x)\ast \beta (y)$
for all $x, y \in Q$. By using the fact that $\alpha$ and $\beta $
are two commuting quadri-algebra endomorphisms, one can compute, for
all $x, y, z \in Q$:
\begin{eqnarray*}
&&(x\nwarrow _{(\alpha , \beta )}y)\nwarrow _{(\alpha , \beta )}\beta (z)=
(\alpha ^2(x)\nwarrow \alpha \beta (y))\nwarrow \beta ^2(z), \\
&&(x\nearrow _{(\alpha , \beta )}y)\nwarrow _{(\alpha , \beta )}\beta (z)=
(\alpha ^2(x)\nearrow \alpha \beta (y))\nwarrow \beta ^2(z), \\
&&\alpha (x)\nwarrow _{(\alpha , \beta )}(y\ast _{(\alpha , \beta )}z)=
\alpha ^2(x)\nwarrow(\alpha \beta (y)\ast \beta ^2(z)), \\
&&\alpha (x)\nearrow _{(\alpha , \beta )}(y\prec _{(\alpha , \beta )}z)=
\alpha ^2(x)\nearrow(\alpha \beta (y)\prec \beta ^2(z)).
\end{eqnarray*}
Thus, (\ref{BiHomqua11}) follows from
(\ref{1.19}) applied to the elements $\alpha ^2(x)$,
$\alpha \beta (y)$, $\beta ^2(z)$.
\end{proof}
\begin{remark}
More generally, let $(Q, \nwarrow, \swarrow, \nearrow, \searrow,
\alpha , \beta )$ be a BiHom-quadri-algebra and $\tilde{\alpha },
\tilde{\beta }: Q\rightarrow Q$ two morphisms of
BiHom-quadri-algebras such that any two of the maps $\alpha, \beta ,
\tilde{\alpha }, \tilde{\beta }$ commute. Define new multiplications
on $Q$ by
$x\nearrow 'y=\tilde{\alpha }(x)\nearrow \tilde{\beta }(y)$, 
$x\searrow 'y=\tilde{\alpha }(x)\searrow \tilde{\beta }(y)$, 
$x\nwarrow 'y=\tilde{\alpha }(x)\nwarrow \tilde{\beta }(y)$ and 
$x\swarrow 'y=\tilde{\alpha }(x)\swarrow \tilde{\beta }(y)$. 
Then $(Q,  \nwarrow ', \swarrow ', \nearrow ', \searrow ',\alpha \circ \tilde{\alpha }, \beta
\circ \tilde{\beta })$ is a BiHom-quadri-algebra.
\end{remark}
\begin{remark} Let $(Q, \nwarrow, \swarrow, \nearrow, \searrow,
\alpha , \beta )$ be a BiHom-quadri-algebra. 
From (\ref{BiHomqua11})-(\ref{BiHomqua15}) we obtain, for all $x, y, z\in Q$:
\begin{eqnarray*}
 &&(x\prec y)\prec \beta(z)=\alpha(x) \prec (y\ast z),~  (x\succ y)\prec \beta(z)=\alpha(x) \succ (y\prec
 z),\\
 && (x\ast y)\succ \beta(z)=\alpha(x) \succ (y\succ z).
\end{eqnarray*}
Thus, $(Q, \prec, \succ, \alpha , \beta )$ is a
BiHom-dendriform algebra. By analogy with \cite{agui}, 
we denote it by $Q_h$ and call it the 
horizontal BiHom-dendriform algebra associated to $Q$.

Also from (\ref{BiHomqua11})-(\ref{BiHomqua15}) we obtain, for all $x, y, z\in Q$:
\begin{eqnarray*}
 &&(x\wedge y)\wedge \beta(z)=\alpha(x) \wedge (y\ast z),~  (x\vee y)\wedge \beta(z)=\alpha(x) \vee (y\wedge
 z),\\
 && (x\ast y)\vee \beta(z)=\alpha(x) \vee (y\vee z).
\end{eqnarray*}
Thus, $(Q, \wedge , \vee , \alpha , \beta )$ is a
BiHom-dendriform algebra, which, again by analogy with \cite{agui}, is denoted by $Q_v$ and is called the 
vertical BiHom-dendriform algebra associated to $Q$.
\end{remark}

From Corollary \ref{BHDBHA} we immediately obtain:
\begin{corollary}\label{BHQBHAs}
Let $(Q, \nwarrow, \swarrow, \nearrow, \searrow, \alpha , \beta )$
be a BiHom-quadri-algebra. Then $(Q, \ast, \alpha , \beta)$ is a
BiHom-associative algebra, where 
$x\ast y=x\searrow y+ x\nearrow y+ x\swarrow y+
x\nwarrow y$, for all $x, y\in Q$. 
\end{corollary}

Just as in the classical situation in \cite{agui}, the tensor product of two BiHom-dendriform algebras becomes 
naturally a BiHom-quadri-algebra.
\begin{proposition} \label{TENSOR}
Let $(A, \prec , \succ , \alpha , \beta )$ and $(B, \prec' , \succ'
, \alpha' , \beta' )$ be two BiHom-dendriform algebras. On 
the linear space $A\ot B$ define bilinear operations by (for all
$ a_1, a_2\in A$ and $b_1, b_2\in B$):
\begin{eqnarray*}
&& (a_1\ot b_1)\nwarrow (a_2\ot b_2)=(a_1\prec a_2)\ot (b_1\prec'
b_2),\\
&& (a_1\ot b_1)\swarrow (a_2\ot b_2)=(a_1\prec a_2)\ot (b_1\succ'
b_2),\\
&& (a_1\ot b_1)\nearrow (a_2\ot b_2)=(a_1\succ a_2)\ot (b_1\prec'
b_2),\\
&& (a_1\ot b_1)\searrow (a_2\ot b_2)=(a_1\succ a_2)\ot (b_1\succ'
b_2).
\end{eqnarray*}
Then $(A\ot B, \nwarrow, \swarrow, \nearrow, \searrow, \alpha
\ot \alpha' , \beta\ot \beta' )$ is a BiHom-quadri-algebra.
\end{proposition}
\begin{proof}
The relations (\ref{BiHomqua6})-(\ref{BiHomqua10}) are obvious. 
We denote by $ \succ _{\ot }, \prec _{\ot }, \vee _{\ot }, \wedge _{\ot }, \ast _{\ot }$ 
the operations defined on $A\ot B$ by (\ref{BiHomqua1})-(\ref{BiHomqua5}) corresponding to
 the operations $\nwarrow, \swarrow, \nearrow, \searrow $ defined above. One can easily see that, 
for all
$ a_1, a_2\in A$ and $b_1, b_2\in B$, we have 
\begin{eqnarray*}
&& (a_1\ot b_1)\prec_{\ot} (a_2\ot b_2)=(a_1\prec a_2)\ot (b_1\ast'
b_2),\\
&& (a_1\ot b_1)\succ_{\ot} (a_2\ot b_2)=(a_1\succ a_2)\ot (b_1\ast'
b_2),\\
&& (a_1\ot b_1)\wedge_{\ot} (a_2\ot b_2)=(a_1\ast a_2)\ot (b_1\prec'
b_2),\\
&& (a_1\ot b_1)\vee_{\ot} (a_2\ot b_2)=(a_1\ast a_2)\ot (b_1\succ'
b_2),\\
&& (a_1\ot b_1)\ast_{\ot} (a_2\ot b_2)=(a_1\ast a_2)\ot (b_1\ast'
b_2).
\end{eqnarray*}
Now we prove (\ref{BiHomqua11}) and leave the rest to the reader:\\[2mm]
${\;\;\;\;\;\;\;\;\;\;\;}$
$[(a_1\ot b_1)\nwarrow (a_2\ot b_2)]\nwarrow (\beta\ot
\beta')(a_3\ot b_3)$
\begin{eqnarray*}
&=& [(a_1\prec a_2)\ot (b_1\prec' b_2)]\nwarrow (\beta(a_3)\ot
\beta'(b_3))\\
&=&((a_1\prec a_2)\prec \beta(a_3))\ot ((b_1\prec' b_2)\prec'
\beta'(b_3))\\
&\overset{(\ref{BiHomdend6})}{=}& (\alpha(a_1)\prec(a_2\ast a_3))\ot (\alpha'(b_1)\prec'(b_2\ast'
b_3))\\
&=& (\alpha(a_1)\ot \alpha'(b_1))\nwarrow [(a_2\ast a_3)\ot
(b_2\ast' b_3)]\\
&=& (\alpha\ot \alpha')(a_1\ot b_1)\nwarrow [(a_2\ot b_2)\ast_{\ot}
(a_3\ot b_3)],
\end{eqnarray*}
${\;\;\;\;\;\;\;\;\;\;\;}$
$[(a_1\ot b_1)\nearrow(a_2\ot b_2)]\nwarrow (\beta\ot
\beta')(a_3\ot b_3)$
\begin{eqnarray*}
&=& [(a_1\succ a_2)\ot (b_1\prec' b_2)]\nwarrow (\beta(a_3)\ot
\beta'(b_3))\\
&=&((a_1\succ a_2)\prec \beta(a_3))\ot ((b_1\prec' b_2)\prec'
\beta'(b_3))\\
&\overset{(\ref{BiHomdend6}),\;(\ref{BiHomdend7})}{=}& 
(\alpha(a_1)\succ(a_2\prec a_3))\ot (\alpha'(b_1)\prec'(b_2\ast'
b_3))\\
&=& (\alpha(a_1)\ot \alpha'(b_1))\nearrow [(a_2\prec a_3)\ot
(b_2\ast'
b_3)]\\
&=& (\alpha\ot \alpha')(a_1\ot b_1)\nearrow [(a_2\ot b_2)\prec_{\ot}
(a_3\ot b_3)],
\end{eqnarray*}
finishing the proof.
\end{proof}
%%%%%%%%%%%%%%%%%%%%%%%%%%%%%%%%%%%%%%
\section{Rota-Baxter operators}\label{sec4}
%%%%%%%%%%%%%%%%%%%%%%%%%%%%%
\setcounter{equation}{0}
%%%%%%%%%%%%%%%%%%%%%%%%%%%%
In this section, we study Rota-Baxter structures in the case of BiHom-type algebras. 
The definition of a Rota-Baxter operator for a BiHom-type algebra is exactly the same as in Definition \ref{RB-DEF}. 
Notice that it  only uses the multiplication but not the structure maps. 
We first construct some examples, then provide some constructions and properties.
\begin{example} We consider the following 2-dimensional BiHom-associative algebra (introduced in \cite{gmmp}), 
where the multiplication and the structure maps   $\alpha$, $%
\beta $  are defined, with respect to a basis  $\{e_1,e_2\}$,  by
%\begin{eqnarray*}
%&& \alpha_1(e_1)=e_1, \;\;\;\alpha_1(e_2)=\frac{2 a}{b-1}e_1-e_2, \\
%&& \beta_1(e_1)=e_1, \;\;\;\beta_1(e_2)=-ae_1+b e_2, \\
%&& \mu_1(e_1,e_1)=e_1, \;\;\;\mu_1(e_1,e_2)= -ae_1+b e_2, \\
%&& \mu_1(e_2,e_1)=\frac{2 a}{b-1}e_1-e_2, \;\;\; \mu_1(e_2,e_2)=-\frac{
%a^2(b-2)}{(b-1)^2}e_1+ae_2,
%\end{eqnarray*}
%and
\begin{align*}
& \mu(e_1,e_1)=e_1, && \mu(e_1,e_2)= be_1+(1-a) e_2, \\
& \mu(e_2,e_1)=\frac{b(1-a)}{a}e_1+ae_2,  && \mu(e_2,e_2)=\frac{ b}{a}
e_2,\\
& \alpha(e_1)=e_1,  && \alpha(e_2)=\frac{b(1-a)}{a}e_1+a e_2, \\
& \beta(e_1)=e_1, && \beta(e_2)=be_1+(1-a) e_2, 
\end{align*}
where $a,b$ are parameters in ${\Bbbk}$ with  $a\neq 0$.

\begin{itemize}
\item They carry  Rota-Baxter operators $R$ of weight 0, defined with respect to the basis by 
$$R(e_1)=0, \ \ 
R(e_2)= r e_1,
$$
or
$$R(e_1)=r_1 e_1+r_2 e_2, \ \ 
R(e_2)=-\frac{ r_1^2}{r_2}e_1-r_1e_2.
$$

\item They carry  the following Rota-Baxter operators $R$ of weight 1, defined with respect to the basis by 
$$R(e_1)=-e_1, \ \ 
R(e_2)= r e_1,
$$
or
$$R(e_1)=0, \ \ 
R(e_2)= r e_1-e_2,
$$
or
$$R(e_1)=-e_1, \ \ 
R(e_2)= -e_2,
$$
or
$$R(e_1)=r_1 e_1+r_2 e_2, \ \ 
R(e_2)=-\frac{ r_1(r_1+1)}{r_2}e_1-(r_1+1)e_2. 
$$
In these formulae, $r, r_1,r_2$ are parameters in ${\Bbbk}$ with $r_2\neq 0$. 
\end{itemize}
\end{example}
\begin{proposition}
Let $A$ be a linear space, $\mu :A\ot A\rightarrow A$ a linear multiplication on $A$, let $R:A\rightarrow A$ be a Rota-Baxter operator of weight $\lambda $ for $(A, \mu )$ and $\alpha , \beta :A\rightarrow A$ two linear maps such that
$R\circ \alpha =\alpha \circ R$ and $R\circ \beta =\beta \circ R$. Define a new multiplication on $A$ by $x*y=\alpha (x)\beta (y)$, for all $x, y\in A$.
Then $R$ is also a Rota-Baxter operator of weight
$\lambda $ for $(A, *)$.
In particular, if $(A, \mu )$ is associative and $\alpha , \beta $ are commuting algebra endomorphisms, then
$R$ is a Rota-Baxter operator of weight $\lambda $ for the BiHom-associative algebra
$A_{(\alpha , \beta )}=
(A, \mu \circ (\alpha \ot \beta ), \alpha , \beta )$.
\end{proposition}
\begin{proof} We compute:
\begin{eqnarray*}
R(x)*R(y)&=&\alpha (R(x))\beta (R(y))=R(\alpha (x))R(\beta (y))\\
&=&R(R(\alpha (x))\beta (y)+\alpha (x)R(\beta (y))+\lambda \alpha (x)\beta (y))\\
&=&R(\alpha (R(x))\beta (y)+\alpha (x)\beta (R(y))+\lambda \alpha (x)\beta (y))\\
&=&R(R(x)*y+x*R(y)+\lambda x*y),
\end{eqnarray*}
finishing the proof.
\end{proof}
\begin{proposition}\label{BHRB}
Let $(A, \mu , \alpha , \beta )$ be a BiHom-associative algebra and $R:A\rightarrow A$ a Rota-Baxter operator
of weight $\lambda $ commuting with $\alpha $ and $\beta $. Define the operations
$\prec $, $\succ $ and $\cdot $ on $A$ by
$x\prec y=xR(y)$, $x\succ y=R(x)y$ and $x\cdot y=\lambda xy$, 
for all $x, y\in A$. Then $(A, \prec , \succ , \cdot , \alpha , \beta )$ is a BiHom-tridendriform algebra.
\end{proposition}
\begin{proof}
We only prove (\ref{BiHomtridend8})-(\ref{BiHomtridend10}) and leave the rest to the reader. We have:\\[2mm]
${\;\;\;}$$(x\prec y)\prec \beta (z)-\alpha (x)\prec (y\prec z)-\alpha (x)\prec (y\succ z)-\alpha (x)\prec (y\cdot z)$
\begin{eqnarray*}
&=&(xR(y))\prec \beta (z)-\alpha (x)\prec (yR(z))-\alpha (x)\prec (R(y)z)-\alpha (x)\prec (\lambda yz)\\
&=&(xR(y))R(\beta (z))-\alpha (x)R(yR(z))-\alpha (x)R(R(y)z)-\alpha (x)R(\lambda yz)\\
&=&(xR(y))\beta (R(z))-\alpha (x)R(yR(z))-\alpha (x)R(R(y)z)-\alpha (x)R(\lambda yz)\\
&=&\alpha (x)(R(y)R(z))-\alpha (x)R(yR(z)+R(y)z+\lambda yz)=0,
\end{eqnarray*}
${\;\;\;}$$(x\succ y)\prec \beta (z)-\alpha (x)\succ (y\prec z)$
\begin{eqnarray*}
&=&(R(x)y)\prec \beta (z)-\alpha (x)\succ (yR(z))
=(R(x)y)R(\beta (z))-R(\alpha (x))(yR(z))\\
&=&(R(x)y)\beta (R(z))-\alpha (R(x))(yR(z))
=\alpha (R(x))(yR(z))-\alpha (R(x))(yR(z))=0,
\end{eqnarray*}
${\;\;\;}$$\alpha (x)\succ (y\succ z)-(x\prec y)\succ \beta (z)-(x\succ y)\succ \beta (z)-(x\cdot y)\succ \beta (z)$
\begin{eqnarray*}
&=&\alpha (x)\succ (R(y)z)-(xR(y))\succ \beta (z)-(R(x)y)\succ \beta (z)-(\lambda xy)\succ \beta (z)\\
&=&R(\alpha (x))(R(y)z)-R(xR(y))\beta (z)-R(R(x)y)\beta (z)-R(\lambda xy)\beta (z)\\
&=&\alpha (R(x))(R(y)z)-(R(x)R(y)\beta (z)
=(R(x)R(y))\beta (z)-(R(x)R(y)\beta (z)=0,
\end{eqnarray*}
finishing the proof.
\end{proof}
\begin{corollary} \label{BHRBzero}
Let $(A, \mu , \alpha , \beta )$ be a BiHom-associative algebra and $R:A\rightarrow A$ a Rota-Baxter operator
of weight 0 commuting with $\alpha $ and $\beta $. Define operations
$\prec $ and $\succ $ on $A$ by
$x\prec y=xR(y)$ and $x\succ y=R(x)y$, 
for all $x, y\in A$. Then $(A, \prec , \succ , \alpha , \beta )$ is a BiHom-dendriform algebra.
\end{corollary}

As a consequence of Propositions \ref{DTDD} and \ref{BHRB}, we obtain:
\begin{proposition}
Let $(A, \mu , \alpha , \beta )$ be a BiHom-associative algebra and $R:A\rightarrow A$ a Rota-Baxter
operator of weight $\lambda $ such that $R\circ \alpha =\alpha \circ R$ and $R\circ \beta =\beta \circ R$.
Define operations $\prec '$ and $\succ '$ on $A$ by
$x\prec 'y=xR(y)+\lambda xy$ and $x\succ 'y=R(x)y$,
for all $x, y\in A$. Then $(A, \prec ', \succ ', \alpha , \beta )$ is a BiHom-dendriform algebra.
\end{proposition}

As a consequence of Propositions \ref{BHTDBH} and \ref{BHRB} we obtain:
\begin{corollary}\label{RBBHAA}
Let $(A, \mu , \alpha , \beta )$ be a BiHom-associative algebra and $R:A\rightarrow A$ a Rota-Baxter operator
of weight $\lambda $ such that $R\circ \alpha =\alpha \circ R$ and $R\circ \beta =\beta \circ R$. If we define
on $A$ a new multiplication by
$x*y=xR(y)+R(x)y+\lambda xy$,
for all $x, y\in A$, then $(A, * , \alpha , \beta )$ is a BiHom-associative algebra.
\end{corollary}

Inspired by \cite{agui}, we introduce the following concept:
\begin{definition}
Let $(D, \prec , \succ , \alpha , \beta )$ be a BiHom-dendriform
algebra. A Rota-Baxter operator of weight 0 on $D$ is a linear map $R:
D\rightarrow D$ such that $R\circ \alpha =\alpha \circ R$, 
$R\circ \beta =\beta \circ R$ and the following conditions are satisfied, for all $x, y\in D$:
\begin{eqnarray}
&& R(x)\succ R(y)=R(x\succ R(y)+ R(x)\succ y),  \label{baxter1}\\
&& R(x)\prec R(y)=R(x\prec R(y)+ R(x)\prec y).  \label{baxter2}
\end{eqnarray}
\end{definition}

From Corollary \ref{BHDBHA} we know that $(D, *, \alpha , \beta )$ is a
BiHom-associative algebra. Adding the equations (\ref{baxter1}) and
(\ref{baxter2}) we obtain that $R$ is also a Rota-Baxter operator of
weight 0 for $(D, \ast)$:
\begin{eqnarray*}
&& R(x)\ast R(y)=R(x\ast R(y)+ R(x)\ast y).  \label{baxter3}
\end{eqnarray*}
\begin{proposition} \label{operation}
Let $(D, \prec , \succ , \alpha , \beta )$ be a BiHom-dendriform
algebra and $R: D\rightarrow D$ a Rota-Baxter operator of weight 0 for $D$. 
Define new operations on $D$ by
\begin{eqnarray*}
x\searrow_R y=R(x)\succ y, ~ x\nearrow_R y=x\succ R(y), ~
x\swarrow_R y=R(x)\prec y ~and ~ x\nwarrow_R y=x\prec R(y).
\end{eqnarray*}
Then $(D, \nwarrow_R, \swarrow_R, \nearrow_R, \searrow_R, \alpha ,
\beta )$ is a BiHom-quadri-algebra.
\end{proposition}
\begin{proof}
We only check (\ref{BiHomqua12}) and leave the rest to the reader. 
We denote by $ \succ _{R }, \prec _{R}, \vee _{R}, \wedge _{R}, \ast _{R}$ 
the operations defined on $D$ by (\ref{BiHomqua1})-(\ref{BiHomqua5}) corresponding to
 the operations $\nwarrow _R, \swarrow _R, \nearrow _R, \searrow _R$, which are then defined by 
\begin{eqnarray*}
x\prec_R y&=&x\nwarrow_R y+x\swarrow_R y=x\prec R(y)+R(x)\prec y, \label{derR1}\\
x\succ_R y&=&x\searrow_R y+x\nearrow_R y=R(x)\succ y+x\succ R(y), \label{derR2}\\
x\wedge_R y&=&x\nearrow_R y+x\nwarrow_R y=x\succ R(y)+x\prec R(y), \label{derR3}\\
x\vee_R y&=&x\searrow_R y+x\swarrow_R y=R(x)\succ y+R(x)\prec y, \label{derR4} \\
x\ast _Ry&=&x\nwarrow_R y+x\swarrow_R y+x\searrow_R y+x\nearrow_R y \nonumber \\
&=&
x\prec R(y)+R(x)\prec y+R(x)\succ y+x\succ R(y).
\end{eqnarray*}

For all $x, y, z\in D$ we have:
\begin{eqnarray*}
(x\wedge_R y)\nearrow_R \beta(z)&=&(x\nwarrow_R y+ x\nearrow_R y)\succ R(\beta(z))\\
&=&(x\prec R(y)+ x\succ R(y))\succ \beta(R(z))\\
&\overset{(\ref{BiHomdend8})}{=}&\alpha(x)\succ(R(y)\succ R(z))\\
&\overset{(\ref{baxter1})}{=}&\alpha(x)\succ R(y\succ R(z)+ R(y)\succ z)\\
&=& \alpha(x)\succ R(y\nearrow_R z+y\searrow_R z)
=\alpha(x)\nearrow_R (y\succ_R z),
\end{eqnarray*}
\begin{eqnarray*}
(x\swarrow_R y)\nwarrow_R \beta(z)&=&(R(x)\prec y)\prec R(\beta(z))
=(R(x)\prec y)\prec \beta(R(z))\\
&\overset{(\ref{BiHomdend6})}{=}&\alpha(R(x))\prec (y\prec R(z)+ y\succ R(z))\\
&=&R(\alpha(x))\prec (y\nwarrow_R z+ y\nearrow_R z)
= \alpha(x)\swarrow_R (y\wedge_R z),
\end{eqnarray*}
 as needed.
\end{proof}
\begin{remark} In the setting of Proposition \ref{operation}, the axioms (\ref{baxter1}) and (\ref{baxter2}) can be
rewritten as 
$R(x)\succ R(y)=R(x\succ_R y)$ and $R(x)\prec R(y)=R(x\prec_R y)$. 
Thus, $R$ is a morphism of BiHom-dendriform algebras from $D_h=(D,
\prec_R, \succ_R, \alpha , \beta)$ to $(D, \prec, \succ,
\alpha , \beta)$.

On the other hand, if we denote by $(D, \ast , \alpha , \beta )$ the BiHom-associative algebra obtained 
from $D$ as in Corollary \ref{BHDBHA}, it is obvious that we have 
$x\wedge_R y=x\ast R(y)$ and 
$x\vee_R y=R(x)\ast y$, for all $x, y\in D$. Thus, the BiHom-dendriform algebra structure 
obtained on $D$ by applying Corollary \ref{BHRBzero} for the Rota-Baxter operator $R$ on the 
BiHom-associative algebra $(D, \ast , \alpha , \beta )$ is exactly the vertical BiHom-dendriform algebra 
$D_v=(D, \wedge_R, \vee _R, \alpha , \beta)$ obtained from the BiHom-quadri-algebra 
$(D, \nwarrow_R, \swarrow_R, \nearrow_R, \searrow_R, \alpha ,
\beta )$.
\end{remark}

Similarly to the classical case in \cite{agui}, examples of BiHom-quadri-algebras are provided by pairs of commuting 
Rota-Baxter operators on BiHom-associative algebras. 
\begin{proposition} \label{pairRB}
Let $R$ and $P$ be a pair of commuting Rota-Baxter operators of
weight 0 on a BiHom-associative algebra $(A, \mu, \alpha , \beta)$ such that 
$R\circ \alpha =\alpha \circ R$, $R\circ \beta =\beta \circ R$,
$P\circ \alpha =\alpha \circ P$ and $P\circ \beta =\beta \circ P$.
Then $P$ is a Rota-Baxter operator of weight 0 on the
BiHom-dendriform algebra $(A, \prec_R, \succ_R, \alpha , \beta)$
corresponding to $R$ as in Corollary \ref{BHRBzero}.
\end{proposition}
\begin{proof}
We check the axioms (\ref{baxter1}) and (\ref{baxter2}). For all
$x, y\in A$, we have:
\begin{eqnarray*}
P(x)\succ_R P(y)&=& R(P(x))P(y)=P(R(x))P(y)\\
&\overset{(\ref{RBrel})}{=}&P(R(x)P(y)+P(R(x))y)=P(R(x)P(y)+R(P(x))y)\\
&=& P(x\succ_R P(y)+ P(x)\succ_R y),
\end{eqnarray*}
\begin{eqnarray*}
P(x)\prec_R P(y)&=& P(x)R(P(y))=P(x)P(R(y))\\
&\overset{(\ref{RBrel})}{=}&P(x P(R(y))+P(x)R(y))=P(x R(P(y))+P(x)R(y))\\
&=& P(x\prec_R P(y)+ P(x)\prec_R y),
\end{eqnarray*}
 as needed.
\end{proof}
\begin{corollary}\label{corpairRB}
In the setting of Proposition \ref{pairRB}, there exists 
a BiHom-quadri-algebra structure on the underlying linear 
space $(A, \alpha , \beta)$, with operations defined by
\begin{eqnarray*}
&& x\searrow y=P(x)\succ_R y=P(R(x))y=R(P(x))y,\\
&& x\nearrow y=x\succ_R P(y)=R(x)P(y),\\
&& x\swarrow y=P(x)\prec_R y=P(x)R(y),\\
&& x\nwarrow y=x\prec_R P(y)=xR(P(y))=xP(R(y)). 
\end{eqnarray*}
In particular, $(A, \ast , \alpha , \beta )$ is a BiHom-associative algebra, where 
$a\ast b=R(P(a))b+R (a)P (b)+P(a)R(b)+aR(P(b))$, for all $a, b\in A$. 
\end{corollary}
\begin{proof}
Apply Proposition \ref{operation} to the Rota-Baxter operator $P$ of
weight 0 on the BiHom-dendriform algebra $(A, \prec_R, \succ_R,
\alpha , \beta)$.
\end{proof}
%%%%%%%%%%%%%%%%%%%%%%%%%%%%%%%%%%%%%%
\section{Free Rota-Baxter BiHom-associative algebra, categories and functors}\label{sec6}
%%%%%%%%%%%%%%%%%%%%%%%%%%%%%
\setcounter{equation}{0}
%%%%%%%%%%%%%%%%%%%%%%%%%%%%%%%%%%%%%%
We define the free BiHom-nonassociative algebra and free Rota-Baxter BiHom-associative algebra, generalizing the construction provided first by D. Yau for Hom-nonassociative algebras in \cite{YauEnv} and then extended to Rota-Baxter Hom-associative algebras in \cite{makhloufyau}. A variation, in the multiplicative case, was presented in \cite{CLG-M-T}. 
See also \cite{martini} for the free  BiHom-nonassociative algebra construction.

\subsection{BiHom-modules and BiHom-nonassociative algebras}

 We denote by  $\BiHomMod$  the category of BiHom-modules. An object in this category is a triple $(M,\alpha_M,\beta_M)$ consisting of a $\Bbbk$-linear space $M$ and  linear maps $\alpha_M \colon M \to M$ and $\beta_M \colon M \to M$ satisfying $\beta_M \circ \alpha_M= \alpha_M\circ\beta_M$.  A morphism $f \colon (M,\alpha_M,\beta_M) \longrightarrow (N,\alpha_N,\beta_M)$ of BiHom-modules is a linear map $f \colon M \to N$ such that $f \circ \alpha_M = \alpha_N \circ f$ and $f \circ \beta_M = \beta_N \circ f$.
 
There is a forgetful functor $U : \BiHomMod \rightarrow  $$\;_{\Bbbk }\mathcal{M}$ that sends a BiHom-module 
$(M,\alpha ,\beta)$ to the linear space $M$, forgetting about the maps $\alpha$ and $\beta$. Conversely, using a similar construction as in \cite{YauEnv}, one may construct a free BiHom-module associated to a linear space. Hence, we have an adjunction.
 
We call  BiHom-nonassociative algebra a quadruple $(A, \mu, \alpha, \beta)$, where $(A,\alpha, \beta)$ is a BiHom-module and 
$\mu  : A \otimes  A \rightarrow A$ is a linear map, called the multiplication of $A$, for which 
$\alpha$ and $\beta$ are algebra maps.

Let $(A, \mu, \alpha, \beta)$ and $(A', \mu', \alpha', \beta')$ be two BiHom-nonassociative algebras. A morphism of 
BiHom-nonassociative algebras from $(A, \mu, \alpha, \beta)$ to $(A', \mu', \alpha', \beta')$ is a BiHom-module morphism $f : A \rightarrow A'$ such that $f\circ \mu=\mu \circ (f\otimes f)$.
 We denote the category of BiHom-nonassociative algebras by $\BiHomNonAs$.

\subsection{B-augmented and RB-augmented planar trees}
The construction of  the  free BiHom-nonassociative algebra and respectively the free Rota-Baxter BiHom-nonassociative algebra involve planar trees and B-augmented trees (respectively RB-augmented trees) 
in order to freely generate products (for $\mu$) and images of $R$. 

For any natural number $n\geq 1$, let $T_n$ denote the set of planar binary trees with $n$ leaves and one root. Below are the sets $T_n$ for $n=1,2,3,4$:
$$
T_1=\left\{ \mbox{\begin{tikzpicture}[xscale=0.4, yscale=0.4,baseline={([yshift=-.8ex]current bounding box.center)}]
%Lines
\draw[line width=1pt] (0,0) -- (0,2);
\end{tikzpicture}} \: \right\}, \quad T_2= \left\{ \mbox{
\begin{tikzpicture}[xscale=0.4, yscale=0.4,baseline={([yshift=-.8ex]current bounding box.center)}]
\draw[line width=1pt] (0,0) -- (0,1) -- (1,2);
\draw[line width=1pt] (0,1) -- (-1,2);
\end{tikzpicture}
} \right\} , \quad T_3=\left\{ \mbox{
\begin{tikzpicture}[xscale=0.4, yscale=0.4,baseline={([yshift=-.8ex]current bounding box.center)}]
\draw[line width=1pt] (0,0) -- (0,1) -- (1,2);
\draw[line width=1pt] (0.5,1.5) -- (0,2);
\draw[line width=1pt] (0,1) -- (-1,2);
\end{tikzpicture}
}, \mbox{
\begin{tikzpicture}[xscale=0.4, yscale=0.4,baseline={([yshift=-.8ex]current bounding box.center)}]
\draw[line width=1pt] (0,0) -- (0,1) -- (1,2);
\draw[line width=1pt] (0,1) -- (-1,2);
\draw[line width=1pt] (-0.5,1.5) -- (0,2);
\end{tikzpicture}
} \right\} $$
$$ T_4=\left\{ \mbox{
\begin{tikzpicture}[xscale=0.4, yscale=0.4,baseline={([yshift=-.8ex]current bounding box.center)}]
\draw[line width=1pt] (0,0) -- (0,1) -- (1,2);
\draw[line width=1pt] (0.3,1.3) -- (-0.2,2);
\draw[line width=1pt] (0.6,1.6) -- (0.35,2);
\draw[line width=1pt] (0,1) -- (-1,2);
\end{tikzpicture}
}, \mbox{
\begin{tikzpicture}[xscale=0.4, yscale=0.4,baseline={([yshift=-.8ex]current bounding box.center)}]
\draw[line width=1pt] (0,0) -- (0,1) -- (1,2);
\draw[line width=1pt] (0.3,1.3) -- (-0.2,2);
\draw[line width=1pt] (0.07,1.6) -- (0.45,2);
\draw[line width=1pt] (0,1) -- (-1,2);
\end{tikzpicture}
},\mbox{
\begin{tikzpicture}[xscale=0.4, yscale=0.4,baseline={([yshift=-.8ex]current bounding box.center)}]
\draw[line width=1pt] (0,0) -- (0,1) -- (1,2);
\draw[line width=1pt] (0,1) -- (-1,2);
\draw[line width=1pt] (0.6,1.6) -- (0.2,2);
\draw[line width=1pt] (-0.6,1.6) -- (-0.2,2);
\end{tikzpicture}
}, \mbox{
\begin{tikzpicture}[xscale=0.4, yscale=0.4,baseline={([yshift=-.8ex]current bounding box.center)}]
\draw[line width=1pt] (0,0) -- (0,1) -- (1,2);
\draw[line width=1pt] (-0.3,1.3) -- (0.2,2);
\draw[line width=1pt] (-0.6,1.6) -- (-0.35,2);
\draw[line width=1pt] (0,1) -- (-1,2);
\end{tikzpicture}
}, \mbox{
\begin{tikzpicture}[xscale=0.4, yscale=0.4,baseline={([yshift=-.8ex]current bounding box.center)}]
\draw[line width=1pt] (0,0) -- (0,1) -- (1,2);
\draw[line width=1pt] (-0.3,1.3) -- (0.2,2);
\draw[line width=1pt] (-0.07,1.6) -- (-0.45,2);
\draw[line width=1pt] (0,1) -- (-1,2);
\end{tikzpicture}
}\right\}.
$$
%An element  $\varphi \in T_n$ shall be called an $n$-tree for short. When necessary we label the leaves of an $n$-tree by $1,2,3, \dots, n$ from left to right.
%%The cardinal of $T_{n+1}$ is the Catalan number $c_n=\frac{(2n)!}{n!(n+1)!}$.
%
%Let  $\varphi \in T_n$ and $\psi \in T_m$, the $(n+m)$-tree $\varphi  \vee \psi$, called the \emph{grafting of $\varphi$ and $\psi$}, is obtained by joining the roots of $\varphi $ and $\psi$ to create a new root. For instance,
%\begin{center}
%\begin{tikzpicture}[xscale=0.4, yscale=0.4]
%\draw[line width=1pt] (0,0) -- (0,1) -- (1,2);
%\draw[line width=1pt] (0,1) -- (-1,2);
%\draw[line width=1pt] (-0.5,1.5) -- (0,2);
%\draw (2,1) node {$\vee$};
%\end{tikzpicture}
% \begin{tikzpicture}[xscale=0.4, yscale=0.4]
%\draw[line width=1pt] (0,0) -- (0,1) -- (1,2);
%\draw[line width=1pt] (0,1) -- (-1,2);
%\draw (2,1) node {$=$};
%\end{tikzpicture}
%\begin{tikzpicture}[xscale=0.3, yscale=0.3]
%\draw[line width=1pt] (1,0) -- (0,1) -- (1,2);
%\draw[line width=1pt] (0,1) -- (-1,2);
%\draw[line width=1pt] (-0.5,1.5) -- (0,2);
%\draw[line width=1pt] (1,0) -- (1,-1);
%\draw[line width=1pt] (1,0) -- (3,2);
%\draw[line width=1pt] (2.5,1.5) -- (2,2);
%\end{tikzpicture}
%\end{center}

A \textbf{B-augmented $n$-tree}  is a pair $(\varphi, a)$, where 
$\varphi \in T_n$ is an $n$-tree and $a$ is an $n$-tuple $(a_1, a_2, \dots,a_n)$ consisting of $n$ pairs of nonnegative integers.

A \textbf{RB-augmented $n$-tree} is a triple $(\varphi, a,f)$ where $(\varphi, a)$ is a  B-augmented $n$-tree  and $f$ is a function  that assigns a nonnegative integer to each element $v$ in $N(\varphi)$, the set of nodes of the tree $\varphi$.

We call the tree $\varphi$ the underlying tree of the B-augmented $n$-tree $(\varphi,a)$ (respectively   RB-augmented $n$-tree  $(\varphi,a,f)$) while,
for all $ i=1, \dots,n$, the pair  of integers $a_i$ is called the \textbf{($\alpha,\beta$)-power} of the leaf $i$ and $f(v)$ is the    \textbf{$R$-power}  of the node $v$.
As we have labelled the $n$ leaves of a tree $\varphi \in T_n$, we can consider that $a$ is a function from the set $\{1,2, \dots, n\}$ to  $\mathbb{N}^2$.
%A leaf weighted $n$-tree can then be seen as a $(n+1)$-tuple $(\varphi, a_1,a_2, \dots, a_n)$ with $\varphi \in T_n$ and $a_1, a_2, \dots,a_n$ non-negative integers.

We denote by $BT_n$ (respectively $\overline{BT}_n$) the set of B-augmented (respectively RB-augmented) planar binary trees with $n$ leaves. The unions over $n\in\mathbb{N}$ of the sets $T_n$, $BT_n$, $\overline{BT}_n$ are denoted respectively $T$, $BT$ and $\overline{BT}$.
We depict the following  examples of B-augmented (respectively RB-augmented)  $3$-trees:
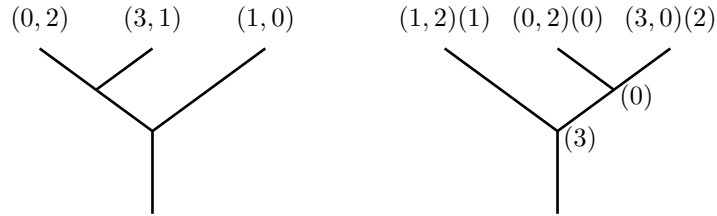
\begin{center}
\begin{tikzpicture}[xscale=0.75, yscale=0.55]
%Lines
\draw[line width=1pt] (0,0) -- (0,2) -- (2,4);
\draw[line width=1pt] (0,2) -- (-2,4);
\draw[line width=1pt] (-1,3) -- (0,4);

%Stations
%\draw (0,0) node {$\bullet$};

%\draw (0,2) node {$\bullet$};
%\draw (-1,3) node {$\bullet$};
\draw (-2,4.7) node {\makebox(0,0){\small$(0,2)$}};
\draw (0,4.7) node {\makebox(0,0){\small$(3,1)$}};
\draw (2,4.7) node {\makebox(0,0){\small$(1,0)$}};
%\draw (-2,4) node {$\bullet$};
%\draw (0,4) node {$\bullet$};
%\draw (2,4) node {$\bullet$};
%\draw (0,-1) node {(a) Tree 1};
\end{tikzpicture}\hspace*{2cm}
\begin{tikzpicture}[xscale=0.75, yscale=0.55]

%Lines
\draw[line width=1pt] (0,0) -- (0,2) -- (-2,4);
\draw[line width=1pt] (0,2) -- (2,4);
\draw[line width=1pt] (1,3) -- (0,4);
\draw (-2,4.7) node {\makebox(0,0){\small$(1,2)(1)$}};
\draw (0,4.7) node {\makebox(0,0){\small$(0,2)(0)$}};
\draw (2,4.7) node {\makebox(0,0){\small$(3,0)(2)$}};
\draw (0.4,1.8) node {\makebox(0,0){\small$(3)$}};
\draw (1.4,2.8) node {\makebox(0,0){\small$(0)$}};
\end{tikzpicture}
\end{center}

Given an $n$-tree (respectively B-augmented or RB-augmented $n$-tree)  $\psi$ and an $m$-tree (respectively 
B-augmented or RB-augmented $m$-tree) $\varphi$, their \textbf{grafting}
$
\psi \vee \varphi \in T_{n+m} $ (respectively in $BT_{n+m}$ or $\overline{BT}_{n+m}$) 
is the tree obtained by joining the roots of $\psi $ and $\varphi $ to create a new root (in the RB-augmented case, one puts label 0 on the new node). Pictorially, 
\mbox{
\begin{tikzpicture}[xscale=0.4, yscale=0.4,baseline={([yshift=-.8ex]current bounding box.center)}]
\draw[line width=1pt] (0,0) -- (0,1) -- (1,2);
\draw[line width=1pt] (0,1) -- (-1,2);
\draw (1, 2.5) node {\makebox(0,0){\small$\varphi$}};
\draw (-1,2.5) node {\makebox(0,0){\small$\psi$}};
\end{tikzpicture}
}.
For example:
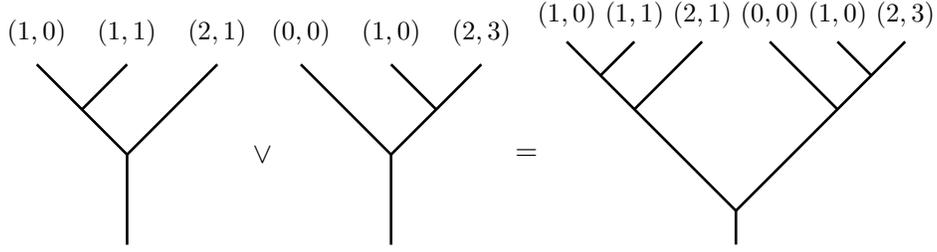
\begin{center}
\begin{tikzpicture}[xscale=0.6, yscale=0.6]
%Lines
\draw[line width=1pt] (0,0) -- (0,2) -- (2,4);
\draw[line width=1pt] (0,2) -- (-2,4);
\draw[line width=1pt] (-1,3) -- (0,4);

\draw (-2,4.7) node {\makebox(0,0){\small$(1,0)$}};
\draw (0,4.7) node {\makebox(0,0){\small$(1,1)$}};
\draw (2,4.7) node{\makebox(0,0){\small$(2,1)$}};
\draw (3,2) node {$\vee$};
\end{tikzpicture}
\begin{tikzpicture}[xscale=0.6, yscale=0.6]
%Lines
\draw[line width=1pt] (0,0) -- (0,2) -- (-2,4);
\draw[line width=1pt] (0,2) -- (2,4);
\draw[line width=1pt] (1,3) -- (0,4);

%Stations
%\draw (0,2) node {$\bullet$};
%\draw (1,3) node {$\bullet$};
\draw (-2,4.7) node {\makebox(0,0){\small$(0,0)$}};
\draw (0,4.7) node {\makebox(0,0){\small$(1,0)$}};
\draw (2,4.7) node {\makebox(0,0){\small$(2,3)$}};
\draw (3,2) node {$=$};
%\draw (0,-1) node {(b) Tree 2};
\end{tikzpicture}
\begin{tikzpicture}[xscale=0.9, yscale=0.9]
%Lines
\draw[line width=1pt] (0,1) -- (0,1.5) -- (-1.5,3) -- (-2.5,4);
\draw[line width=1pt] (-2,3.5) -- (-1.5,4);
\draw[line width=1pt] (-1.5,3) -- (-0.5,4);
\draw[line width=1pt] (0,1.5) -- (1.5,3) -- (2.5,4);
\draw[line width=1pt] (1.5,3) -- (0.5,4);
\draw[line width=1pt] (2,3.5) -- (1.5,4);

%Stations
%\draw (0,0) node {$\bullet$};

\draw (-2.5,4.4) node {\makebox(0,0){\small$(1,0)$}};
\draw (-1.5,4.4) node {\makebox(0,0){\small$(1,1)$}};
\draw (-0.5,4.4) node {\makebox(0,0){\small$(2,1)$}};
\draw (0.5,4.4) node {\makebox(0,0){\small$(0,0)$}};
\draw (1.5,4.4) node {\makebox(0,0){\small$(1,0)$}};
\draw (2.5,4.4) node {\makebox(0,0){\small$(2,3)$}};
\end{tikzpicture}
\end{center}

Note that grafting is neither an associative nor a commutative operation. For any tree $\varphi \in T_n$ 
(respectively B-augmented), there are unique integers $p$ and $q$ with $p+q=n$ and trees 
$\varphi_1 \in T_p$ (respectively in $BT_p$) and $\varphi_2 \in T_q$  (respectively in $BT_q$) such that $\varphi= \varphi_1 \vee \varphi_2$. It is clear that any tree in $T_n$  (respectively in $BT_n$) can be obtained from \begin{tikzpicture}[xscale=0.4, yscale=0.4,baseline={([yshift=-.8ex]current bounding box.center)}]
%Lines
\draw[line width=1pt] (0,0) -- (0,1);
\end{tikzpicture} , the 1-tree  (respectively  B-augmented 1-tree),  by successive grafting. For the B-augmented  trees $(\varphi ,a)$, one has $(\varphi,a)= (\varphi_1 ,a_{(1)})\vee (\varphi_2,a_{(2)})$, where $a_{(1)},a_{(2)}$ are the corresponding leaf weights. For the RB-augmented trees one needs in addition to use the map $R$ described in \eqref{RR}, that is there exists moreover an integer $s$ such that  $\varphi=R^s( \varphi_1 \vee \varphi_2)$ with $\varphi_1 \in \overline{BT}_{p}$ and $\varphi_2 \in \overline{BT}_{q}$. For example: 
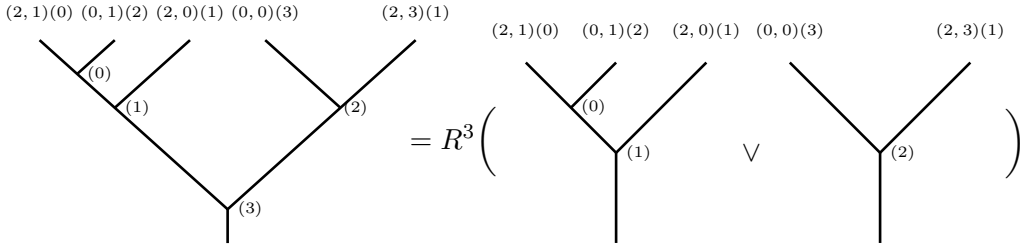
\begin{center}
\begin{tikzpicture}[xscale=1, yscale=0.9]
%Lines
\draw[line width=1pt] (0,1) -- (0,1.5) -- (-1.5,3) -- (-2.5,4);
\draw[line width=1pt] (-2,3.5) -- (-1.5,4);
\draw[line width=1pt] (-1.5,3) -- (-0.5,4);
\draw[line width=1pt] (0,1.5) -- (1.5,3) -- (2.5,4);
\draw[line width=1pt] (1.5,3) -- (0.5,4);
%\draw[line width=1pt] (2,3.5) -- (1.5,4);

%Stations
%\draw (0,0) node {$\bullet$};
\draw (0.3,1.5) node {\makebox(0,0){\tiny$(3)$}};
\draw (-1.2,3) node {\makebox(0,0){\tiny$(1)$}};
\draw (1.7,3) node {\makebox(0,0){\tiny$(2)$}};
\draw (-1.7,3.5) node {\makebox(0,0){\tiny$(0)$}};

\draw (-2.5,4.4) node {\makebox(0,0){\tiny$(2,1)(0)$}};
\draw (-1.5,4.4) node {\makebox(0,0){\tiny$(0,1)(2)$}};
\draw (-0.5,4.4) node {\makebox(0,0){\tiny$(2,0)(1)$}};
\draw (0.5,4.4) node {\makebox(0,0){\tiny$(0,0)(3)$}};
%\draw (1.5,4.4) node {\makebox(0,0){\small$(1,0)$}};
\draw (2.5,4.4) node {\makebox(0,0){\tiny$(2,3)(1)$}};
\draw (3,2.5) node {$= R^3 \bigg($};
\end{tikzpicture}
\begin{tikzpicture}[xscale=0.6, yscale=0.6]
%Lines
\draw[line width=1pt] (0,0) -- (0,2) -- (2,4);
\draw[line width=1pt] (0,2) -- (-2,4);
\draw[line width=1pt] (-1,3) -- (0,4);

\draw (-2,4.7) node {\makebox(0,0){\tiny$(2,1)(0)$}};
\draw (0,4.7) node {\makebox(0,0){\tiny$(0,1)(2)$}};
\draw (2,4.7) node{\makebox(0,0){\tiny$(2,0)(1)$}};
\draw (3,2) node {$\vee$};

\draw (0.5,2) node {\makebox(0,0){\tiny$(1)$}};
\draw (-0.5,3) node {\makebox(0,0){\tiny$(0)$}};
\end{tikzpicture}
\begin{tikzpicture}[xscale=0.6, yscale=0.6]
%Lines
\draw[line width=1pt] (0,0) -- (0,2) -- (-2,4);
\draw[line width=1pt] (0,2) -- (2,4);
%\draw[line width=1pt] (1,3) -- (0,4);

%Stations
%\draw (0,2) node {$\bullet$};
%\draw (1,3) node {$\bullet$};
\draw (-2,4.7) node {\makebox(0,0){\tiny$(0,0)(3)$}};
%\draw (0,4.7) node {\makebox(0,0){\small$(1,0)$}};
\draw (2,4.7) node {\makebox(0,0){\tiny$(2,3)(1)$}};
\draw (0.5,2) node {\makebox(0,0){\tiny$(2)$}};
\draw (3,2.2) node {$\bigg)$};
%\draw (0,-1) node {(b) Tree 2};
\end{tikzpicture}
\end{center}

%
%For all $n \geq 1$, we let $B_n$ denote the set of leaf weighted $n$-trees. Let $B$ denote the union over $n \in {\mathbb N}$ of the sets $B_n$ together with an element that we call the unit and denote by $\mathbb{1}$. Note that the element $\mathbb{1}$ is different from the leaf weighted 1-tree $\begin{tikzpicture}[baseline={([yshift=-.8ex]current bounding box.center)}]
%\draw[line width=1pt] (0,0) -- (0,0.4);
%\draw (0,0.6) node {$0$};
%\end{tikzpicture}$. 
We  consider the linear spaces ${\mathbb{BT}}$ freely generated by the set~$BT$ and ${\overline{\mathbb{BT}}}$ freely 
generated by the set~$\overline{BT}$. 
The grafting is extended linearly to ${\mathbb{BT}}$ and ${\overline{\mathbb{BT}}}$, on which we define moreover   two  linear maps $\alpha$ and $\beta$ and  one more extra  linear map $R$ on ${\overline{\mathbb{BT}}}$, in the following way:\\[2mm]
$\bullet $ The map  $\alpha:{\mathbb{BT}} \to {\mathbb{BT}} $ (respectively $\alpha:{\overline{\mathbb{BT}}} \to {\overline{\mathbb{BT}}} $) 
sends a B-augmented (respectively RB-augmented) $n$-tree to a B-augmented (respectively RB-augmented) $n$-tree obtained  
by adding $+1$ to all  first components of $a_i$ in $a$, i.e. 
 \begin{small}
 \begin{equation}\label{alphaF}
 \alpha((\varphi, ((a_{11},a_{12}), (a_{21},a_{22}),\cdots, (a_{n1},a_{n2}))))=(\varphi, ((a_{11}+1,a_{12}), (a_{21}+1,a_{22}),\cdots, (a_{n1}+1,a_{n2})))
 \end{equation}
 \end{small}
 respectively
  \begin{small}
 \begin{equation}\label{alphaFR}
 \alpha((\varphi, ((a_{11},a_{12}), (a_{21},a_{22}),\cdots, (a_{n1},a_{n2}))),f)=(\varphi, ((a_{11}+1,a_{12}), (a_{21}+1,a_{22}),\cdots, (a_{n1}+1,a_{n2})),f)
 \end{equation}
 \end{small}
$\bullet $ The map $\beta:{\mathbb{BT}} \to {\mathbb{BT}} $   (respectively  $\beta:{\overline{\mathbb{BT}}} \to {\overline{\mathbb{BT}}} $) 
sends a B-augmented (respectively RB-augmented) $n$-tree to a B-augmented (respectively RB-augmented) 
$n$-tree obtained by  adding $+1$ to all  second components of $a_i$ in $a$, i.e. 
  \begin{small}
 \begin{equation}\label{betaF}
 \beta((\varphi, ((a_{11},a_{12}), (a_{21},a_{22}),\cdots, (a_{n1},a_{n2}))))=(\varphi, ((a_{11},a_{12}+1), (a_{21},a_{22}+1),\cdots, (a_{n1},a_{n2}+1)))
  \end{equation}
  \end{small}
  respectively
   \begin{small}
 \begin{equation}\label{betaFR}
 \beta((\varphi, ((a_{11},a_{12}), (a_{21},a_{22}),\cdots, (a_{n1},a_{n2})),f))=(\varphi, ((a_{11},a_{12}+1), (a_{21},a_{22}+1),\cdots, (a_{n1},a_{n2}+1)),f)
  \end{equation}
  \end{small}
$\bullet $ The map $R:{\overline{\mathbb{BT}}} \to {\overline{\mathbb{BT}}} $ sends a RB-augmented $n$-tree $(\varphi,a,f)$ to a  RB-augmented $n$-tree  $(\varphi,a,Rf)$ such that    $R f$ is equal to $f$, except that for the lowest vertex $v_{low}$  (or the only leaf if $\varphi \in T_1$) we have \begin{equation}Rf(v_{low})=f(v_{low})+1.\label{RR}\end{equation}

%For example,% the image of Tree 1 by $\al$ is
%$$\alpha\left( \mbox{\begin{tikzpicture}[xscale=0.2, yscale=0.2,baseline={([yshift=-.8ex]current bounding box.center)}]
%
%%Lines
%\draw[line width=1pt] (0,0) -- (0,2) -- (2,4);
%\draw[line width=1pt] (0,2) -- (-2,4);
%\draw[line width=1pt] (-1,3) -- (0,4);
%
%%Stations
%%\draw (0,0) node {$\bullet$};
%
%%\draw (0,2) node {$\bullet$};
%%\draw (-1,3) node {$\bullet$};
%\draw (-2,4.8) node {$0$};
%\draw (0,4.8) node {$2$};
%\draw (2,4.8) node {$1$};
%%\draw (-2,4) node {$\bullet$};
%%\draw (0,4) node {$\bullet$};
%%\draw (2,4) node {$\bullet$};
%%\draw (0,-1) node {(a) Tree 1};
%\end{tikzpicture}} \right)= \mbox{ \begin{tikzpicture}[xscale=0.2, yscale=0.2,baseline={([yshift=-.8ex]current bounding box.center)}]
%
%%Lines
%\draw[line width=1pt] (0,0) -- (0,2) -- (2,4);
%\draw[line width=1pt] (0,2) -- (-2,4);
%\draw[line width=1pt] (-1,3) -- (0,4);
%
%%Stations
%\draw (-2,4.8) node {$1$};
%\draw (0,4.8) node {$3$};
%\draw (2,4.8) node {$2$};
%\end{tikzpicture}}
%$$

Observe that the maps $\alpha,\ \beta, \ R$ defined above commute and $\alpha,\ \beta$ are algebra morphisms for the grafting of trees, that is
$\alpha((\varphi,a) \vee (\psi,b))= \alpha((\varphi,a))\vee \alpha((\psi,b))$ and 
$\beta((\varphi,a) \vee (\psi,b))= \beta((\varphi,a))\vee \beta((\psi,b))$, 
for  $(\varphi,a), (\psi,b)$ in  $ \mathbb{BT}$ (same for  $\overline{\mathbb{BT}}$).

\subsection{Free BiHom-nonassociative algebra }
 
Let $(A, \mu, \alpha, \beta)$ be a BiHom-nonassociative algebra.  Given elements $x_1,\cdots, x_n\in A$, there are $\#T_n = C_{n-1} $ (Catalan's number) ways to parenthesize the monomial $x_1 \cdots x_n$ to obtain an element in $A$. Hence, every parenthesized monomial $x_1 \cdots x_n$ corresponds to  an $n$-tree. Each B-augmented $n$-tree provides a way to multiply $n$ elements in a BiHom-nonassociative algebra $(A, \mu, \alpha, \beta)$. We define maps 
\begin{equation}\label{actionBT}
{\mathbb{BT}}\otimes A^{\otimes n}\rightarrow A, \quad ((\varphi,a);x_1\otimes \cdots \otimes x_n)\rightarrow (\alpha^{a_{11}}\beta^{a_{12}}(x_1)\cdots \alpha^{a_{n1}}\beta^{a_{n2}}(x_n))_\varphi
\end{equation}
inductively via the rules:\\
1) $(x)_I=x\ \text{ for } x\in A$, where $I$ denotes the 1-tree.\\
2) If $\varphi=\varphi_1\vee \varphi_2$, where $\varphi_1\in T_p$ and $ \varphi_2 \in T_q$,  then 
 $(x_1 \cdots x_n)_\varphi=(x_1 \cdots x_p)_{\varphi_1}(x_{p+1} \cdots x_{p+q})_{\varphi_2}.$

Recall that $(x_1\cdots x_n)_\varphi$ means the element obtained by putting  $x_1, \cdots , x_n$ on the leaves of the $n$-tree $\varphi$ and applying the multiplication with respect to the tree.

There is a forgetful functor $E : \BiHomNonAs \rightarrow  \BiHomMod$ that sends a BiHom-nonassociative algebra  $(A,\mu,\alpha ,\beta)$ to the BiHom-module $(A,\alpha ,\beta)$, forgetting about the multiplication $\mu$. Conversely, using a similar construction as in \cite{YauEnv}, one may prove, by using the previous construction, for the functor $E$, the existence of a left adjoint 
$F_{BHNAs} : \BiHomMod \rightarrow \BiHomNonAs   $ defined as $$F_{BHNAs} (M)=\bigoplus_{n\geq 1}
\bigoplus_{(\varphi , a)\in BT_n}M_{(\varphi , a)}^{\otimes n},$$
for $(M,\alpha,\beta)\in \BiHomMod$, where $M_{(\varphi , a)}^{\otimes n}$ is a copy of $M^{\otimes n}$ 
(its generators will be denoted by $(x_1 \otimes \cdots \otimes x_n)_{(\varphi , a)}$). 

Hence, $F_{BHNAs}(M)$ is called the free BiHom-nonassociative algebra of the BiHom-module $(M,\alpha,\beta)$, where the multiplication $\mu_F$  is defined by grafting and the structure linear maps $\alpha_F$ and $\beta_F$ are defined in \eqref{alphaF} and \eqref{betaF}.
\begin{remark}
One may get the free BiHom-associative algebra and the enveloping algebra of a BiHom-Lie algebra by considering quotients over ideals constructed in a similar way as in \cite{YauEnv} and by adding the images of the second structure map. The free BiHom-associative algebra is obtained by a quotient with a two-sided ideal $I^\infty$ defined as 
$$I^1=\langle im(\mu_F\circ (\mu_F\otimes \beta _F-\alpha _F\otimes \mu_F)) \rangle, \quad I^{n+1}=\langle I^n\cup \alpha_F(I^n) \cup \beta_F(I^n)\rangle, \quad I^ \infty=\cup_{n\geq 1}I^n.$$
\end{remark}

\subsection{The category of Rota-Baxter BiHom-associative algebras}

We denote by  $\BiHomRB_\lambda$  the  category of Rota-Baxter BiHom-associative algebras of weight $\lambda$.  A Rota-Baxter BiHom-associative algebra of weight $\lambda$ is a tuple $(A,\mu,\alpha,\beta,R)$ in which
$(A,\mu,\alpha,\beta)$ is a BiHom-associative algebra, 
the linear map $R \colon A \to A$ satisfies the Rota-Baxter identity \eqref{RBrel} and 
$\alpha \circ R = R \circ \alpha$ and $\beta \circ R = R \circ \beta$.

A morphism $f \colon (A,\mu_A,\alpha_A,\beta_A,R_A) \to (B,\mu_B,\alpha_B,\beta_B,R_B)$ of Rota-Baxter BiHom-associative algebras of weight $\lambda$ is a morphism $f \colon (A,\alpha_A,\beta_A) \to (B,\alpha_B,\beta_B)$ of the underlying BiHom-modules such that $f \circ \mu_A = \mu_B \circ f^{\otimes 2}$ and $f \circ R_A = R_B \circ f$.

A Rota-Baxter algebra $(A,\mu,R)$ of weight $\lambda$ gives rise to a Rota-Baxter BiHom-associative algebra of weight 
$\lambda$ with 
$\alpha = \beta = id_A$.  A morphism of Rota-Baxter algebras of weight $\lambda$ is an algebra morphism that commutes with the Rota-Baxter operators $R$.  In particular, the functor $(A,\mu,R) \mapsto (A,\mu, id_A, id_A, R)$ from the category of Rota-Baxter algebras of weight $\lambda$ to the category $\BiHomRB_\lambda$ is a full and faithful embedding.  So we can regard the category of Rota-Baxter algebras and Rota-Baxter Hom-associative algebras (for which $\alpha=\beta$)  as  subcategories of  $\BiHomRB_\lambda$.

Similarly to  \cite{makhloufyau} and using B-augmented and RB-augmented $n$-trees, we may get:
\begin{theorem}
\label{thm:freehrb}
The forgetful functor
$\cO \colon \BiHomRB_\lambda \to \BiHomMod$ admits a right adjoint, where 
$\cO(A,\mu,\alpha,\beta,R) = (A,\alpha, \beta)$. 
\end{theorem}

The proof uses  an intermediate category $\bD$, whose objects are tuples $(A,\mu,\alpha,\beta,R)$, where
$A$ is a $\Bbbk$-linear space, $\mu:A\otimes A\rightarrow A$ is a multiplication and $\alpha,\beta,R $ are commuting two by two linear self-maps $A\rightarrow A$.
A morphism $f \colon (A,\mu_A,\alpha_A,\beta_A,R_A) \to (B,\mu_B,\alpha_B,\beta_B,R_B)$ in $\mathbf{D}$ consists of a linear map $f \colon A \to B$ such that $f \circ \mu_A = \mu_B \circ f^{\otimes 2}$, $f \circ \alpha_A = \alpha_B \circ f$, $f \circ \beta_A = \beta_B \circ f$ and $f \circ R_A = R_B \circ f$.

  There are forgetful functors
$
\BiHomRB_\lambda \xrightarrow{\cO_2} \bD \xrightarrow{\cO_1} \BiHomMod,
$
whose composition is $\cO$.  One shows that each of these two forgetful functors $\cO_i$ admits a left adjoint $\cF_i$.  The composition
$
\cF = \cF_2 \circ \cF_1
$
is then the desired left adjoint.  

First, consider the  forgetful functor
   $
   \cO_1 \colon \bD \to \BiHomMod
  $
defined as
$
\cO_1(A,\mu,\alpha,\beta,R) = (A,\alpha,\beta).
$
The functor $\cO_1$ admits a left adjoint
$
\cF_1 \colon \BiHomMod \to \mathbf{D},
$
defined as follows. If $(M, \alpha , \beta )\in \BiHomMod$, we set $\cF_1 (M)=\bigoplus_{n\geq 1}\bigoplus_{(\varphi , a,f)\in \overline{BT}_n}M_{(\varphi , a,f)}^{\otimes n},$ where $M_{(\varphi , a,f)}^{\otimes n}$ is a copy of $M^{\otimes n}$ 
(its generators will be denoted by $(x_1\otimes ...\otimes x_n)_{(\varphi , a, f)}$). We need to define on $\cF_1 (M)$ 
the four operations $\mu , \alpha , \beta , R$. The multiplication $\mu $ on $\cF_1 (M)$ is defined by 
$$\mu((x_1 \otimes \cdots \otimes x_n)_{(\varphi , a,f)},(x_{n+1} \otimes \cdots \otimes x_{n+m})_{(\psi , b,g)})=(x_1 \otimes \cdots \otimes x_{n+m})_{(\varphi , a,f)\vee  (\psi , b,g)}.$$
 The maps $\alpha,\beta$ and $R$ are defined, using  \eqref{alphaFR}, \eqref{betaFR}, \eqref{RR}, by 
\begin{eqnarray*}
&& \alpha((x_1 \otimes \cdots \otimes x_n)_{(\varphi , a,f)})=(x_1 \otimes \cdots \otimes x_n)_{\alpha(\varphi , a,f)}, \\
 && \beta((x_1 \otimes \cdots \otimes x_n)_{(\varphi , a,f)})=(x_1 \otimes \cdots \otimes x_n)_{\beta(\varphi , a,f)}, \\
&& R((x_1 \otimes \cdots \otimes x_n)_{(\varphi , a,f)})=(x_1 \otimes \cdots \otimes x_n)_{R(\varphi , a,f)}.
 \end{eqnarray*}

In order to prove that $\cF_1$ is the left adjoint of the forgetful functor $ \cO_1$, one has to define and use a certain action of RB-augmented trees. Namely, if $(A, \mu , \alpha , \beta , R)\in \bD$, define  ${\overline{\mathbb{BT}}}\otimes A^{\otimes n}\rightarrow A$, 
\begin{eqnarray}\label{actionTT}
&&((\varphi,a,f);x_1\otimes \cdots \otimes x_n)\rightarrow (\alpha^{a_{11}}\beta^{a_{12}}R^{f(v_{top,1})}(x_1)\cdots \alpha^{a_{n1}}\beta^{a_{n2}}R^{f(v_{top,n})}(x_n))_\varphi , 
\end{eqnarray}
where $v_{top,i}$ are the highest nodes (leaves). Moreover, at each internal node $v$ the map $R$ is applied with power $f(v)$. 
For example, the following tree applied to $x_1\otimes x_2\otimes x_3$ 

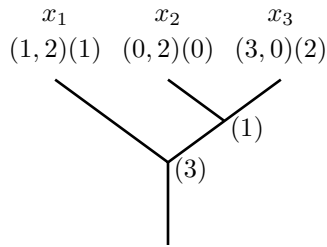
\begin{center}
%\begin{tikzpicture}[xscale=0.75, yscale=0.55]
%%Lines
%\draw[line width=1pt] (0,0) -- (0,2) -- (2,4);
%\draw[line width=1pt] (0,2) -- (-2,4);
%\draw[line width=1pt] (-1,3) -- (0,4);
%
%%Stations
%%\draw (0,0) node {$\bullet$};
%
%%\draw (0,2) node {$\bullet$};
%%\draw (-1,3) node {$\bullet$};
%\draw (-2,4.7) node {\makebox(0,0){\small$(0,2)$}};
%\draw (0,4.7) node {\makebox(0,0){\small$(3,1)$}};
%\draw (2,4.7) node {\makebox(0,0){\small$(1,0)$}};
%%\draw (-2,4) node {$\bullet$};
%%\draw (0,4) node {$\bullet$};
%%\draw (2,4) node {$\bullet$};
%%\draw (0,-1) node {(a) Tree 1};
%\end{tikzpicture}\hspace*{2cm}
\begin{tikzpicture}[xscale=0.75, yscale=0.55]

%Lines
\draw[line width=1pt] (0,0) -- (0,2) -- (-2,4);
\draw[line width=1pt] (0,2) -- (2,4);
\draw[line width=1pt] (1,3) -- (0,4);
\draw (-2,4.7) node {\makebox(0,0){\small$(1,2)(1)$}};
\draw (-2,5.5) node {\makebox(0,0){\small$x_1$}};
\draw (0,4.7) node {\makebox(0,0){\small$(0,2)(0)$}};
\draw (0,5.5) node {\makebox(0,0){\small$x_2$}};
\draw (2,4.7) node {\makebox(0,0){\small$(3,0)(2)$}};
\draw (0.4,1.8) node {\makebox(0,0){\small$(3)$}};
\draw (2,5.5) node {\makebox(0,0){\small$x_3$}};
\draw (1.4,2.8) node {\makebox(0,0){\small$(1)$}};
\end{tikzpicture}
\end{center}
translates to $R^3(\alpha\beta^2R(x_1)R(\beta^2 (x_2)\alpha^3R^2(x_3)))$.

To construct $\cF_2$, a left adjoint of $\cO_2$, one   picks an object $(A,\mu,\alpha,\beta,R)\in  \mathbf{D}$ and take  $S$ a  subset of $A$  consisting of the generating relations in a Rota-Baxter BiHom-associative algebra, i.e. the elements
 $ \alpha(x)(yz)-(xy)\beta(z)$ and $R(x)R(y)-R(R(x)y+xR(y)+\lambda xy) $, 
for all $x,y,z \in A$. Then we set
$\cF_2(A) = A/\langle S\rangle $.

\subsection{The categories of BiHom-dendriform algebras and BiHom-tridendriform algebras}
Naturally and similarly to what we did above, we consider
$\BiHomdidend$, the category of BiHom-dendriform algebras and 
$\BiHomtridend$, the category of BiHom-tridendriform algebras. One shows that the forgetful functors from $\BiHomdidend$ (respectively $\BiHomtridend$) to $\BiHomMod$, the category of BiHom-modules, admit left adjoints. Moreover, we have:
\begin{theorem} (i) There is an adjoint pair of functors
   \begin{equation}
   \label{eq:HDD}
   U_{\mathcal{BD}} \colon \BiHomdidend \rightleftarrows \BiHomRB_0 \colon \mathcal{BD},
   \end{equation}
in which the right adjoint is given by
$
\mathcal{BD}(A,\mu,\alpha,\beta, R) = (A,\prec,\succ,\alpha,\beta) \in \BiHomdidend
$
with
   $x \prec y =xR(y)$ and 
   $x \succ y =R(x)y$, 
for $x, y \in A$.

(ii) There is an adjoint pair of functors
   \[
   U_{\mathcal{BT}} \colon \BiHomtridend \rightleftarrows \BiHomRB_\lambda \colon \mathcal{BT},
   \]
in which $U_{\mathcal{BT}}$ is the left adjoint.  For $(A,\mu,\alpha,\beta,R) \in \BiHomRB_\lambda$, the binary operations in the object
   $\mathcal{BT}(A) = (A,\prec,\succ,\cdot,\alpha,\beta) \in \BiHomtridend $
are defined as
   $
   x \prec y= xR(y), \ x \succ y = R(x)y, \ x \cdot y = \lambda xy
   $, 
for $x,y \in A$.
\end{theorem}
The proof is similar to the Hom-type case in \cite{makhloufyau}.
%%%%%%%%%%%%%%%%%%%%%%%%%%%%%%%%%%%%%%
\section{Weak BiHom-pseudotwistors}\label{sec8}
%%%%%%%%%%%%%%%%%%%%%%%%%%%%%
\setcounter{equation}{0}
%%%%%%%%%%%%%%%%%%%%%%%%%%%%
Let $(A, \mu )$ be an associative algebra. A weak pseudotwistor for $A$, as defined in \cite{panvan} (extending
the previous proposal from \cite{lpvo} called pseudotwistor)
is a linear
map $T:A\ot A\rightarrow A\ot A$ for which there exists a linear map
$\mathcal{T}:A\otimes A\otimes A \rightarrow A\otimes A\otimes A$ such that
\begin{eqnarray*}
&&T\circ (id_A\otimes (\mu \circ T))=(id_A\otimes \mu )\circ \mathcal{T}, \\
&&T\circ ((\mu \circ T)\otimes id_A)=(\mu \otimes id_A)\circ \mathcal{T}.
\end{eqnarray*}
If this is the case, then $(A, \mu \circ T)$ is also an associative algebra.

If $(A, \mu )$ is an associative algebra and $R:A\rightarrow A$ is a Rota-Baxter operator of weight $\lambda $, then
the linear map
\begin{eqnarray*}
&&T:A\ot A\rightarrow A\ot A, \;\;\;T(a\ot b)=R(a)\ot b+a\ot R(b)+\lambda a\ot b, \;\;\;\forall \;a, b\in A,
\end{eqnarray*}
is a weak pseudotwistor, and consequently we recover the fact that the new multiplication defined on $A$ by
$a*b=R(a)b+aR(b)+\lambda ab$,
for all $a, b\in A$, is associative, see \cite{panvan}.

We want to obtain a BiHom-analogue of this fact. We begin by recalling the following concept and result from \cite{gmmp}:
\begin{proposition}
Let $(D, \mu , \alpha , \beta )$
be a BiHom-associative algebra and $\tilde{\alpha }, \tilde{\beta }:D\rightarrow D$ two
multiplicative linear maps such that any two of the maps $\tilde{\alpha },
\tilde{\beta }, \alpha , \beta $ commute. Let $T:D\otimes D\rightarrow
D\otimes D$ be a linear map and assume that there exist two linear maps $\tilde{%
T}_1, \tilde{T}_2:D\otimes D\otimes D \rightarrow D\otimes D\otimes D$ such that $T$ commutes with
$\alpha \otimes \alpha $, $\beta \otimes \beta $, $\tilde{\alpha }\otimes \tilde{\alpha }$,
$\tilde{\beta }\otimes \tilde{\beta }$ and the following relations hold:
\begin{eqnarray*}
&&T\circ (\alpha \otimes \mu )= (\alpha \otimes \mu )\circ
\tilde{T}_1\circ (T\otimes id_D),  \label{ghompstw1} \\
&&T\circ (\mu \otimes \beta )= (\mu \otimes \beta )\circ
\tilde{T}_2\circ (id_D\otimes T),  \label{ghompstw2} \\
&&\tilde{T}_1\circ (T\otimes id_D)\circ (\tilde{\alpha }\otimes T)= \tilde{T}_2\circ
(id_D\otimes T)\circ (T\otimes \tilde{\beta }).  \label{ghompstw3}
\end{eqnarray*}
Then $D^T_{\tilde{\alpha }, \tilde{\beta }}:=(D, \mu \circ T, \tilde{\alpha }\circ \alpha ,
\tilde{\beta }\circ \beta )$ is also a BiHom-associative algebra. The map $T$
is called an $(\tilde{\alpha }, \tilde{\beta })$-BiHom-pseudotwistor and the two
maps $\tilde{T}_1$, $\tilde{T}_2$ are called the companions of $T$.
In the particular case $\tilde{\alpha }=\tilde{\beta }=id_D$, we call $T$ a
BiHom-pseudotwistor and we denote $D^T_{\tilde{\alpha }, \tilde{\beta }}$ by $D^T$.
\end{proposition}

We introduce now a common generalization of this concept and of the one of weak pseudotwistor for an
associative algebra:
\begin{proposition}
Let $(D, \mu , \alpha , \beta )$
be a BiHom-associative algebra and $\tilde{\alpha }, \tilde{\beta }:D\rightarrow D$ two
multiplicative linear maps such that any two of the maps $\tilde{\alpha },
\tilde{\beta }, \alpha , \beta $ commute. Let $T:D\otimes D\rightarrow
D\otimes D$ be a linear map and assume that there exist a linear map
$\mathcal{T}:D\otimes D\otimes D \rightarrow D\otimes D\otimes D$ such that $T$ commutes with
$\alpha \otimes \alpha $, $\beta \otimes \beta $, $\tilde{\alpha }\otimes \tilde{\alpha }$,
$\tilde{\beta }\otimes \tilde{\beta }$ and the following relations hold:
\begin{eqnarray}
&&T\circ ((\tilde{\alpha }\circ \alpha )\otimes (\mu \circ T))=(\alpha \otimes \mu )\circ \mathcal{T},
\label{WBHps1}\\
&&T\circ ((\mu \circ T)\otimes (\tilde{\beta }\circ \beta ))=(\mu \otimes \beta )\circ \mathcal{T}.
\label{WBHps2}
\end{eqnarray}
Then $D^T_{\tilde{\alpha }, \tilde{\beta }}:=(D, \mu \circ T, \tilde{\alpha }\circ \alpha ,
\tilde{\beta }\circ \beta )$ is also a BiHom-associative algebra. The map $T$
is called a weak $(\tilde{\alpha }, \tilde{\beta })$-BiHom-pseudotwistor and
$\mathcal{T}$ is called the weak companion of $T$.
In the particular case $\tilde{\alpha }=\tilde{\beta }=id_D$, we call $T$ a weak
BiHom-pseudotwistor and we denote $D^T_{\tilde{\alpha }, \tilde{\beta }}$ by $D^T$.
\end{proposition}
\begin{proof}
We only prove the BiHom-associativity condition and leave the rest to the reader:
\begin{eqnarray*}
(\mu \circ T)\circ ((\mu \circ T)\ot (\tilde{\beta }\circ \beta ))&=&\mu \circ T\circ ((\mu \circ T)\ot (\tilde{\beta }\circ \beta ))\\
&\overset{(\ref{WBHps2})}{=}&\mu \circ (\mu \ot \beta )\circ \mathcal{T}
=\mu \circ (\alpha \ot \mu )\circ \mathcal{T}\\
&\overset{(\ref{WBHps1})}{=}&(\mu \circ T)\circ ((\tilde{\alpha }\circ \alpha )\ot (\mu \circ T)),
\end{eqnarray*}
finishing the proof.
\end{proof}
\begin{remark}
If $T$ is an $(\tilde{\alpha }, \tilde{\beta })$-BiHom-pseudotwistor with companions $\tilde{T}_1$,
$\tilde{T}_2$ on a BiHom-associative algebra $(D, \mu , \alpha , \beta )$,
then $T$ is also a weak $(\tilde{\alpha }, \tilde{\beta })$-BiHom-pseudotwistor for $D$, with weak companion
$\mathcal{T}=\tilde{T}_1\circ (T\otimes id_D)\circ (\tilde{\alpha }\otimes T)=
\tilde{T}_2\circ (id_D\otimes T)\circ (T\otimes \tilde{\beta })$.
\end{remark}

We can give now another proof for Corollary \ref{RBBHAA}.
\begin{proposition}
Let $(A, \mu , \alpha , \beta )$ be a BiHom-associative algebra and $R:A\rightarrow A$ a Rota-Baxter operator
of weight $\lambda $ such that $R\circ \alpha =\alpha \circ R$ and $R\circ \beta =\beta \circ R$. Then the
linear map
\begin{eqnarray*}
&&T:A\ot A\rightarrow A\ot A, \;\;\;T(a\ot b)=R(a)\ot b+a\ot R(b)+\lambda a\ot b, \;\;\;\forall \;a, b\in A,
\end{eqnarray*}
is a weak BiHom-pseudotwistor with weak companion $\mathcal{T}:A\otimes A\otimes A
\rightarrow A\otimes A\otimes A$,
\begin{eqnarray*}
&&\mathcal{T}(a\ot b\ot c)=R(a)\ot R(b)\ot c+R(a)\ot b\ot R(c)
+a\ot R(b)\ot R(c)+\lambda R(a)\ot b\ot c\\
&&\;\;\;\;\;\;\;\;\;\;\;\;\;\;\;\;\;\;\;\;\;\;\;\;\;\;\;+\lambda a\ot R(b)\ot c+\lambda a\ot b\ot R(c)
+\lambda ^2a\ot b\ot c,
\end{eqnarray*}
and the new BiHom-associative multiplication $\mu \circ T$ on $A$ is given by
$a*b=R(a)b+aR(b)+\lambda ab$.
\end{proposition}
\begin{proof}
We only prove that $T\circ (\alpha \otimes (\mu \circ T))=(\alpha \otimes \mu )\circ \mathcal{T}$ and leave the rest
to the reader:\\[2mm]
${\;\;\;}$
$T\circ (\alpha \ot (\mu \circ T))(a\ot b\ot c)$
\begin{eqnarray*}
&=&R(\alpha (a))\ot R(b)c+\alpha (a)\ot R(R(b)c)+\lambda \alpha (a)\ot R(b)c\\
&&+R(\alpha (a))\ot bR(c)+\alpha (a)\ot R(bR(c))+\lambda \alpha (a)\ot bR(c)\\
&&+\lambda R(\alpha (a))\ot bc+\lambda \alpha (a)\ot R(bc)+\lambda ^2\alpha (a)\ot bc\\
&=&R(\alpha (a))\ot R(b)c+\lambda \alpha (a)\ot R(b)c+R(\alpha (a))\ot bR(c)+\lambda \alpha (a)\ot bR(c)\\
&&+\lambda R(\alpha (a))\ot bc+\lambda ^2\alpha (a)\ot bc+\alpha (a)\ot R(R(b)c+bR(c)+\lambda bc)\\
&\overset{(\ref{RBrel})}{=}&R(\alpha (a))\ot R(b)c+\lambda \alpha (a)\ot R(b)c+R(\alpha (a))\ot bR(c)+
\lambda \alpha (a)\ot bR(c)\\
&&+\lambda R(\alpha (a))\ot bc+\lambda ^2\alpha (a)\ot bc+\alpha (a)\ot R(b)R(c)\\
&=&\alpha (R(a))\ot R(b)c+\lambda \alpha (a)\ot R(b)c+\alpha (R(a))\ot bR(c)+
\lambda \alpha (a)\ot bR(c)\\
&&+\lambda \alpha (R(a))\ot bc+\lambda ^2\alpha (a)\ot bc+\alpha (a)\ot R(b)R(c)\\
&=&(\alpha \ot \mu )\circ \mathcal{T}(a\ot b\ot c),
\end{eqnarray*}
finishing the proof.
\end{proof}

The next two results are BiHom-analogues of some results in \cite{panvan}.
\begin{proposition}
Let $(A, \mu , \alpha , \beta )$ be a BiHom-associative algebra
and $T, D:A\ot A\rightarrow A\ot A$ two weak BiHom-pseudotwistors for $A$, with weak
companions $\mathcal{T}$ and respectively $\mathcal{D}$, such that the following conditions are satisfied:
\begin{eqnarray}
&&D\circ (\alpha \ot (\mu \circ T\circ D))=(\alpha \ot (\mu \circ T))\circ \mathcal{D}, \label{rel1DT}\\
&&D\circ ((\mu \circ T\circ D)\ot \beta )=((\mu \circ T)\ot \beta )\circ \mathcal{D}. \label{rel2DT}
\end{eqnarray}
Then $T\circ D$ is a weak BiHom-pseudotwistor for $A$, with weak companion $\mathcal{T}\circ \mathcal{D}$.
\end{proposition}
\begin{proof}
Obviously $T\circ D$ commutes with $\alpha \ot \alpha $ and $\beta \ot \beta $.
Now we compute:
\begin{eqnarray*}
T\circ D\circ (\alpha \ot (\mu \circ T\circ D))&\overset{(\ref{rel1DT})}{=}&
T\circ (\alpha \ot (\mu \circ T))\circ \mathcal{D}\\
&\overset{(\ref{WBHps1})}{=}&(\alpha \ot \mu )\circ \mathcal{T}\circ \mathcal{D},
\end{eqnarray*}
\begin{eqnarray*}
T\circ D\circ ((\mu \circ T\circ D)\ot \beta )&\overset{(\ref{rel2DT})}{=}&
T\circ ((\mu \circ T)\ot \beta )\circ \mathcal{D}\\
&\overset{(\ref{WBHps2})}{=}&(\mu \ot \beta )\circ \mathcal{T}\circ \mathcal{D},
\end{eqnarray*}
finishing the proof.
\end{proof}
\begin{corollary}
Let $(A, \mu , \alpha , \beta )$ be a BiHom-associative algebra
and $T, D:A\ot A\rightarrow A\ot A$ two weak BiHom-pseudotwistors for $A$, with weak
companions $\mathcal{T}$ and respectively $\mathcal{D}$, such that the following conditions are satisfied:
\begin{eqnarray}
&&\mu \circ T\circ D=\mu \circ D\circ T, \label{correl1TD}\\
&&\mathcal{D}\circ (id_A\ot T)=(id_A\ot T)\circ \mathcal{D}, \label{correl2TD}\\
&&\mathcal{D}\circ (T\ot id_A)=(T\ot id_A)\circ \mathcal{D}. \label{correl3TD}
\end{eqnarray}
Then $T\circ D$ is a weak BiHom-pseudotwistor for $A$, with weak companion $\mathcal{T}\circ \mathcal{D}$.
\end{corollary}
\begin{proof}
We check (\ref{rel1DT}), while (\ref{rel2DT}) is similar and left to the reader:
\begin{eqnarray*}
D\circ (\alpha \ot (\mu \circ T\circ D))&\overset{(\ref{correl1TD})}{=}&
D\circ (\alpha \ot (\mu \circ D\circ T))\\
&=&D\circ (\alpha \ot (\mu \circ D))\circ (id_A\ot T)\\
&\overset{(\ref{WBHps1})}{=}&(\alpha \ot \mu )\circ \mathcal{D}\circ (id_A\ot T)\\
&\overset{(\ref{correl2TD})}{=}&(\alpha \ot \mu )\circ (id_A\ot T)\circ \mathcal{D}\\
&=&(\alpha \ot (\mu \circ T))\circ \mathcal{D},
\end{eqnarray*}
finishing the proof.
\end{proof}

We can give now another proof for the second statement in Corollary \ref{corpairRB}. 
\begin{corollary}
Let $(A, \mu , \alpha , \beta )$ be a BiHom-associative algebra and $R, P:A\rightarrow A$ two commuting
Rota-Baxter operators of weight $0$ such that $R\circ \alpha =\alpha \circ R$, $R\circ \beta =\beta \circ R$,
$P\circ \alpha =\alpha \circ P$ and $P\circ \beta =\beta \circ P$. Define a new multiplication on $A$ by
\begin{eqnarray}
&&a*b=R(P(a))b+R (a)P (b)+P(a)R(b)+aR(P(b)), \;\;\;
\forall \;a, b\in A.  \label{m2RB}
\end{eqnarray}
Then $(A, *, \alpha , \beta )$ is a BiHom-associative algebra.
\end{corollary}
\begin{proof}
We consider the
weak BiHom-pseudotwistors $T, D:A\ot A\rightarrow A\ot A$,
\begin{eqnarray*}
&&T(a\ot b)=R(a)\ot b+a\ot R(b), \;\;\;D(a\ot b)=P(a)\ot b+a\ot P(b),
\end{eqnarray*}
with weak companions $\mathcal{T}$ and respectively $\mathcal{D}$ defined by
\begin{eqnarray*}
&&\mathcal{T}(a\ot b\ot c)=R(a)\ot R(b)\ot c+R(a)\ot b\ot R(c)+
a\ot R(b)\ot R(c), \\
&&\mathcal{D}(a\ot b\ot c)=P(a)\ot P(b)\ot c+P(a)\ot b\ot P(c)+
a\ot P(b)\ot P(c).
\end{eqnarray*}
Since $R$ and $P$ commute, it is clear that $T\circ D=D\circ T$, so (\ref{correl1TD}) is
satisfied. One can easily check that (\ref{correl2TD}) and (\ref{correl3TD}) are
also satisfied, so $T\circ D$ is a weak BiHom-pseudotwistor and clearly
$\mu \circ T\circ D$ is exactly the multiplication (\ref{m2RB}).
\end{proof}

%%%%%%%%%%%%%%%%%%%%%%%%%%%%%%%%%%%%%%%%%%%%

\begin{center}
ACKNOWLEDGEMENTS
\end{center}

%%%%%%%%%%%%%%%%%%%%%%%%%%%%%%%%%%%%%%%%%%%%%%%
This paper was written while Ling Liu was visiting the Institute of Mathematics of the Romanian Academy (IMAR), 
supported by the NSF of China (Grant Nos. 11601486, 11401534); she would like to thank IMAR for its warm hospitality. 
Claudia Menini was a  member of the National Group for Algebraic
and Geometric Structures, and their Applications (GNSAGA-INdAM).
%%%%%%%%%%%%%%%%%%%%%%%%

\bibliographystyle{amsplain}
%\bibliography{global}
\providecommand{\bysame}{\leavevmode\hbox to3em{\hrulefill}\thinspace}
\providecommand{\MR}{\relax\ifhmode\unskip\space\fi MR }
\providecommand{\MRhref}[2]{%
  \href{http://www.ams.org/mathscinet-getitem?mr=#1}{#2}
}
\providecommand{\href}[2]{#2}

\end{document}